\pdfoutput=1
\documentclass[a4paper,12pt]{article}
\usepackage[text={170mm,254mm},centering]{geometry}

\usepackage[utf8]{inputenc}
\usepackage[T1]{fontenc}

\usepackage{amsmath} 
\usepackage{amsthm}
\usepackage{amsfonts}
\usepackage{bm}

\usepackage{enumitem}

\usepackage{graphicx}
\usepackage{rotating}
\usepackage{url}

\usepackage{tikz}
\usetikzlibrary{calc}
\usetikzlibrary{backgrounds}

\usepackage{authblk}
\usepackage{subfigure}
 
\theoremstyle{plain}
\newtheorem{theorem}{Theorem}[section]
\newtheorem{lemma}[theorem]{Lemma}

\newtheorem{proposition}[theorem]{Proposition}
\newtheorem{conjecture}[theorem]{Conjecture}
\theoremstyle{definition}
\newtheorem{definition}{Definition}[section]

\newtheorem{example}{Example}[section]

\DeclareMathOperator{\Ch}{Ch}
\DeclareMathOperator{\code}{code}
\DeclareMathOperator{\len}{len}
\DeclareMathOperator{\win}{win}
\DeclareMathOperator{\hex}{\mathit h}
\DeclareMathOperator{\mcd}{mcd}
\newcommand{\suma}{{\rm sum}}
\newcommand{\avg}{{\rm avg}}
\newcommand{\cd}{{\rm cd}}
\newcommand{\ex}{{\rm ex}}

\newcommand{\classB}{\mathcal{B}}
\newcommand{\classBcc}{\mathcal{B}^*}
\newcommand{\classBnb}{\mathcal{B}'}
\newcommand{\classF}{\mathcal{F}}
\newcommand{\classFcc}{\mathcal{F}^*}
\newcommand{\classFnb}{\mathcal{F}'}

\title{Convexity Deficit of Benzenoids}

\author[1,2,3]{Nino Ba\v{s}i\'{c}}
\author[4,5]{Sarah Berkemer}
\author[4]{J\"{o}rg Fallmann}
\author[6]{Patrick W. Fowler}
\author[4]{Thomas Gatter}
\author[1,2,3,12]{Toma\v{z} Pisanski}
\author[4,5,7]{Nancy Retzlaff}
\author[4,5,8,9,10]{Peter F. Stadler}
\author[1,11]{Sara S. Zemlji\v{c}}

\affil[1]{FAMNIT, University of Primorska, Koper, Slovenia}
\affil[2]{Institute of Mathematics, Physics and Mechanics, Ljubljana, Slovenia}
\affil[3]{IAM, University of Primorska, Koper, Slovenia}
\affil[4]{Bioinformatics Group, Department for Computer Science, and
  Interdisciplinary Center for Bioinformatics, Leipzig University, Germany}
\affil[5]{Max-Planck-Institute for Mathematics in the Sciences, Leipzig,
  Germany}
\affil[6]{Department of Chemistry, University of Sheffield,
  Sheffield S3 7HF, UK}
\affil[7]{Institute for Infrastructure and Resources Management,
  Leipzig University, Germany}
\affil[8]{Institute for Theoretical Chemistry, University of Vienna, Austria}
\affil[9]{Facultad de Ciencias, Universidad National de Colombia, Sede
  Bogot{\'a}, Colombia}
\affil[10]{Santa Fe Institute, Santa Fe, USA}
\affil[11]{Faculty of Mathematics, Physics and Informatics,
  Comenius University, Bratislava, Slovakia}
\affil[12]{Faculty of Mathematics and Physics,
  University of Ljubljana, Ljubljana, Slovenia}

\date{March 20, 2020}

\begin{document}

\maketitle

\begin{center} 
\textit{This paper is dedicated to the memory of Dr.\ Edward C.\ Kirby (1934--2019)\\
colleague, collaborator and friend of several of the authors.}
\end{center}

\begin{abstract}
  In 2012, a family of benzenoids was introduced by Cruz, Gutman, and Rada,
  which they called \emph{convex benzenoids}. In this paper we introduce
  the \emph{convexity deficit}, a new topological index intended for
  benzenoids and, more generally, fusenes. This index measures by how much
  a given fusene departs from convexity.  It is defined in terms of the
  boundary-edges code.  In particular, convex benzenoids are exactly the
  benzenoids having convexity deficit equal to $0$.  \emph{Quasi-convex}
  benzenoids form the family of non-convex benzenoids that are closest to
  convex, i.e., they have convexity deficit equal to $1$.  Finally, we
  investigate convexity deficit of several important families of
  benzenoids.
\end{abstract}


\section{Introduction}

\emph{Benzenoids} form an important family of graphs and molecules.
Polycyclic (aromatic) hydrocarbons
\cite{clar1964,clar1964a,dias1987,dias1988}, of which the benzenoids form a
subset, are important molecular systems with a rich organic chemistry
\cite{fetzer2000,tomlinson1971} characterised by specific reactivity
\cite{lloyd1989,taylor1990}, spectra \cite{Murrell}, and photophysics.
They occur naturally, geologically and as by-products of natural and
anthropogenic combustion processes, with considerable implications for the
environment \cite{abdel2016} and human health \cite{gelboin1978} and have
been postulated as significant contributors to the carbon inventory in the
wider Universe \cite{allamandola1989}. There has also been a huge amount of
interest over several decades in the graph theory of benzenoids and related
structures and its application to prediction of physical and chemical
properties (see, e.g., the textbooks
\cite{cyvin1994,cyvin1991,cyvin1988,Garratt1986,gutman1989,trinajstic1992}).
Much of the mathematical chemistry literature is concerned with prediction
or rationalisation of electronic structure, but there is also interest in
classification of the shapes available to benzenoids.  As pointed out
before \cite{deza1990}, molecular shape is intimately associated with
molecular electric and steric properties, such as quadrupole moment or van
der Waals envelope, which are implicated in structure-activity
relationships from odour perception \cite{odour1957} to carcinogenicity
\cite{BalKauKosBal1980,gelboin1978,Jerina1980,KnoSzyJerTri1983}.  Codes
based on boundaries seem especially suitable for systematising our notions
of shapes of benzenoids.  The reader is referred to the books
\cite{cyvin_1988,gutman1989} for definitions and basic facts.

\section{Preliminaries}

We begin by giving a mathematical definition of a \emph{fusene}
\cite{Brinkmann2002,Brinkmann2007}.  The class of fusenes contains as a
proper subclass the class of benzenoids.
\begin{definition}
  A \emph{fusene} is a simple subcubic $2$-connected plane graph such that
  all bounded faces are hexagons and all vertices not on the outer face
  have degree $3$.
\end{definition} 
Benzenoids can be now defined in terms of fusenes.
\begin{definition}
  A fusene that can be embedded in the infinite hexagonal lattice is called
  a \emph{benzenoid}.
\end{definition} 
In other words, benzenoids are those fusenes which are also subgraphs of
the infinite hexagonal lattice.

\begin{example}
  Figure~\ref{fig:1} shows four subcubic plane graphs. Pentalene is not a
  fusene, because its bounded faces are pentagons.  Biphenyl is not a
  fusene because it is not $2$-connected. Anthracene and $[6]$helicene are
  both fusenes.  Anthracene is also a benzenoid, whilst $[6]$helicene is
  not.
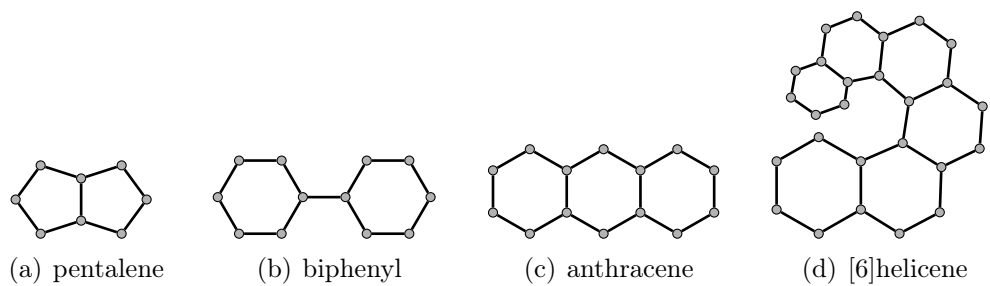
\begin{figure}[!htbp]
\centering
\subfigure[pentalene]{\;\begin{tikzpicture}[scale=-0.4]
\tikzstyle{every node} = [inner sep=1.2, draw, circle,fill=gray!60]
\tikzstyle{edge} = [draw, line width=1.0]
\tikzstyle{periedge} = [draw, line width=1.0]
\pgfmathsetmacro{\k}{360 / 5}
\node (a1) at (0, 0) {};
\node (a2) at ($ (a1) + (90:1.4) $) {};
\node (a3) at ($ (a2) + ({90 + \k}:1.4) $) {};
\node (a4) at ($ (a3) + ({90 + 2*\k}:1.4) $) {};
\node (a5) at ($ (a4) + ({90 + 3*\k}:1.4) $) {};
\node (b3) at ($ (a2) + ({90 - \k}:1.4) $) {};
\node (b4) at ($ (b3) + ({90 - 2*\k}:1.4) $) {};
\node (b5) at ($ (b4) + ({90 - 3*\k}:1.4) $) {};
\draw[edge] (a1) -- (a2) -- (a3) -- (a4) -- (a5) -- (a1);
\draw[edge] (a2) -- (b3) -- (b4) -- (b5) -- (a1);
\end{tikzpicture}\;}
\quad
\subfigure[biphenyl]{\begin{tikzpicture}[scale=0.4]
\tikzstyle{every node} = [inner sep=1.2, draw, circle,fill=gray!60]
\tikzstyle{edge} = [draw, line width=1.0]
\tikzstyle{periedge} = [draw, line width=1.0]
\pgfmathsetmacro{\k}{60}
\node (a1) at (0, 0) {};
\node (a2) at ($ (a1) + (120:1.4) $) {};
\node (a3) at ($ (a2) + (120+\k:1.4) $) {};
\node (a4) at ($ (a3) + (120+2*\k:1.4) $) {};
\node (a5) at ($ (a4) + (120+3*\k:1.4) $) {};
\node (a6) at ($ (a5) + (120+4*\k:1.4) $) {};
\draw[edge] (a1) -- (a2) -- (a3) -- (a4) -- (a5) -- (a6) -- (a1);
\node (b1) at (1.4, 0) {};
\node (b2) at ($ (b1) + (-120+\k:1.4) $) {};
\node (b3) at ($ (b2) + ({-120+2*\k}:1.4) $) {};
\node (b4) at ($ (b3) + ({-120+3*\k}:1.4) $) {};
\node (b5) at ($ (b4) + ({-120+4*\k}:1.4) $) {};
\node (b6) at ($ (b5) + ({-120+5*\k}:1.4) $) {};
\draw[edge] (b1) -- (b2) -- (b3) -- (b4) -- (b5) -- (b6) -- (b1);
\draw[edge] (a1) -- (b1);
\end{tikzpicture}}
\quad
\subfigure[anthracene]{\begin{tikzpicture}[scale=0.4]
\tikzstyle{every node} = [inner sep=1.2, draw, circle,fill=gray!60]
\tikzstyle{edge} = [draw, line width=1.0]
\tikzstyle{periedge} = [draw, line width=1.0]
\node (a1) at (0, 0) {};
\node (a2) at ($ (a1) + (90:1.4) $) {};
\node (a3) at ($ (a2) + (150:1.4) $) {};
\node (a4) at ($ (a3) + (210:1.4) $) {};
\node (a5) at ($ (a4) + (-90:1.4) $) {};
\node (a6) at ($ (a5) + (-30:1.4) $) {};
\node (b1) at ($ (a1) + (-30:1.4) $) {};
\node (b2) at ($ (b1) + (30:1.4) $) {};
\node (b3) at ($ (b2) + (90:1.4) $) {};
\node (b4) at ($ (b3) + (150:1.4) $) {};
\node (c1) at ($ (b2) + (-30:1.4) $) {};
\node (c2) at ($ (c1) + (30:1.4) $) {};
\node (c3) at ($ (c2) + (90:1.4) $) {};
\node (c4) at ($ (c3) + (150:1.4) $) {};
\path[edge] (a1) -- (a2) -- (a3) -- (a4) -- (a5) -- (a6) -- (a1);
\path[edge] (a1) -- (b1) -- (b2) -- (b3) -- (b4) -- (a2);
\path[edge] (b2) -- (c1) -- (c2) -- (c3) -- (c4) -- (b3);
\end{tikzpicture}}
\quad
\subfigure[{$[6]$helicene}]{\label{fig:6helicene}\begin{tikzpicture}[scale=0.4]
\tikzstyle{every node} = [inner sep=1.2, draw, circle,fill=gray!60]
\tikzstyle{edge} = [draw, line width=1.0]
\tikzstyle{periedge} = [draw, line width=1.0]
\node (a1) at (0, 0) {};
\node (a2) at ($ (a1) + (90:1.6) $) {};
\node (a3) at ($ (a2) + (150:1.6) $) {};
\node (a4) at ($ (a3) + (210:1.6) $) {};
\node (a5) at ($ (a4) + (-90:1.6) $) {};
\node (a6) at ($ (a5) + (-30:1.6) $) {};
\path[edge] (a1) -- (a2) -- (a3) -- (a4) -- (a5) -- (a6) -- (a1);
\node (b1) at ($ (a1) + (-30-2:1.5) $) {};
\node (b2) at ($ (b1) + (30-2:1.5) $) {};
\node (b3) at ($ (b2) + (90-2:1.5) $) {};
\node (b4) at ($ (b3) + (150-2:1.5) $) {};
\path[edge] (a1) -- (b1) -- (b2) -- (b3) -- (b4) -- (a2);
\node (c1) at ($ (b3) + (30-4:1.4) $) {};
\node (c2) at ($ (c1) + (90-4:1.4) $) {};
\node (c3) at ($ (c2) + (150-4:1.4) $) {};
\node (c4) at ($ (c3) + (210-4:1.4) $) {};
\path[edge] (b3) -- (c1) -- (c2) -- (c3) -- (c4) -- (b4);
\node (d1) at ($ (c3) + (90-6:1.3) $) {};
\node (d2) at ($ (d1) + (150-6:1.3) $) {};
\node (d3) at ($ (d2) + (210-6:1.3) $) {};
\node (d4) at ($ (d3) + (270-6:1.3) $) {};
\path[edge] (c3) -- (d1) -- (d2) -- (d3) -- (d4) -- (c4);
\node (e1) at ($ (d3) + (150-8:1.1) $) {};
\node (e2) at ($ (e1) + (210-8:1.1) $) {};
\node (e3) at ($ (e2) + (270-8:1.1) $) {};
\node (e4) at ($ (e3) + (330-8:1.1) $) {};
\path[edge] (d3) -- (e1) -- (e2) -- (e3) -- (e4) -- (d4);
\node (f1) at ($ (e3) + (210-10:0.9) $) {};
\node (f2) at ($ (f1) + (270-10:0.9) $) {};
\node (f3) at ($ (f2) + (330-5:1) $) {};
\node (f4) at ($ (f3) + (30-10:1) $) {};
\path[edge] (e3) -- (f1) -- (f2) -- (f3) -- (f4) -- (e4);
\end{tikzpicture}}
\caption{Examples of subcubic plane graphs, two of which are fusenes.}
\label{fig:1}
\end{figure}
\end{example}

Let us denote the class of all benzenoids by $\classB$ and the class of all
fusenses by $\classF$.  The \emph{inner dual} of a plane graph is its dual
graph with the vertex that corresponds to the outer face removed.  A
\emph{catacondensed} fusene is a fusene whose inner dual is a tree. Fusenes
that are not catacondensed are called \emph{pericondensed}. We will denote
the class of all catacondensed fusenes by $\classFcc$.  Catacondensed
fusenes can be further divided into \emph{branched} and \emph{non-branched}
fusenes.  A catacondensed fusene is called non-branched if its inner dual
is a path; otherwise it is called branched. The class of non-branched
fusenes will be denoted by $\classFnb$.  Those definitions are naturally
inherited by benzenoids. The class of catacondensed benzenoids and
non-branched benzenoids will be denoted, respectively, by $\classBcc$ and
$\classBnb$.  In this paper, we restrict to catacondensed benzenoids when
using the terms branched and non-branched.

\subsection{Boundary-edges code revisited}

Each fusene can be assigned a \emph{boundary-edges code} (BEC), a sequence of
numbers counting the number of boundary edges between two vertices of
degree $3$, following the perimeter in an arbitrary, say counter-clockwise,
direction. This useful tool to describe a benzenoid was introduced by
P.~Hansen and his co-workers \cite{hansen1996}. The code depends on the
starting vertex and the chosen direction. However, it can be made unique by
choosing the lexicographically maximal code among all possible codes which
is often called the \emph{canonical code}. Each benzenoid can be
\emph{uniquely} described by such boundary-edges code, but this does not
hold for fusenes \cite{guo2002}. Benzene is an exceptional benzenoid as it
is the only benzenoid with no vertex of degree $3$. If need be, we assign
the code $6$ to benzene. In the present paper the (lexicographically
maximal) boundary-edges code of $B$ will be denoted by $\code(B)$.

Here, we take a different approach and start from the definition of a code:
\begin{definition}
A \emph{code} is a string over the alphabet $\{1,2,3,4,5\}$. 
\end{definition}
Note that we permit codes that are not boundary-edges codes of any
benzenoid (or fusene).

By $c \oplus d$ we denote concatenation of codes $c$ and $d$, e.g.\ 
\begin{equation*}
422 \oplus 5133 = 4225133.
\end{equation*}
Moreover, $\sigma_i(c)$ for
$i \geq 0$ denotes the right circular shift of code $c$ by $i$ positions, e.g.\
\begin{equation*}
\sigma_3(4225133) = 1334225.
\end{equation*}
This operation can also be defined for the negative values of $i$, such
that $\sigma_{-i}(c)$ for $i \geq 0$ is the left circular shift of $c$ by
$i$ positions, e.g.\
\begin{equation*}
\sigma_{-2}(1334225) = 3422513.
\end{equation*}
By $\rho(c)$ we denote the reverse of $c$, e.g.\ 
\begin{equation*}
\rho(3422513) = 3152243.
\end{equation*}
Note that $\rho^2(c) = c$ and $\sigma_i \sigma_{-i}(c) = c$ for every code
$c$.

We will use some properties of codes.
\begin{definition}
  Let $c$ be a code. By $\len(c)$ we denote the length of the code, i.e.\
  the number of symbols appearing in it and by $\suma(c)$ we denote the sum
  of all numbers of the code.  By $\win(c)$ we denote the \emph{winding} of
  the code, which is defined as
\begin{equation*}
  \win(c) = \suma(c) - 2 \len(c).
\end{equation*}
\end{definition}

\begin{lemma}
The winding of a concatenation of two 
codes $c$  and $d$ is additive, i.e. $$\win(c \oplus d) = \win(c) + \win(d).$$ 
\end{lemma}
\begin{proof}
  The length of the code and the sum of the numbers of the code are clearly
  additive. Then
  \begin{align*}
    \win(c \oplus d)  & = \suma(c \oplus  d) - 2 \len(c \oplus  d) \\
                      & = \suma(c) + \suma(d) - 2 \len(c) - 2 \len(d) \\
                      & = \win(c) + \win(d). \qedhere
  \end{align*}
\end{proof}

\begin{definition}
  Two codes $c$ and $d$ are \emph{equivalent} if one can be obtained from
  the other by a circular shift and possibly reversal, i.e.\ if there
  exists an integer $k$ such that $\sigma_k(c) = d$ or
  $\rho \sigma_k(c) = d$.
\end{definition}

\begin{definition}
  A code is \emph{canonical} if it is lexicographically maximal among all
  equivalent codes.
\end{definition}

No simple way is known to check whether a given code is the boundary-edges
code of some benzenoid.  However, there is an obvious necessary condition.
\begin{lemma}
Let $B$ be a benzenoid with at least $2$ hexagons. Then $\win(\code(B)) = 6$.
\end{lemma}
\begin{proof}
  The proof proceeds by induction on the number of hexagons, $h$.

  The only benzenoid with $h = 2$ hexagons is naphthalene, with code
  $55$. Clearly, $\win(55) = 10 - 2 \cdot 2 = 6$.

  It is known that any benzenoid $B$ with $h$ hexagons can be obtained from
  some benzenoid $B'$ with $h - 1$ hexagons by adding a new hexagon using
  either one-, two-, three-, four- or five-contact addition, where
  $k$-contact implies that $k$ edges of the new hexagon are identified with
  $k$ consecutive edges of $B'$ (see \cite[pp.\ 12--13]{gutman1989}).  Let
  $\code(B) = c$ and $\code(B') = c'$. Assume that $\win(c') = 6$.

  If $B$ is obtained by one-contact addition then $c$ can be obtained from
  $c'$ by replacing the symbol $s$ (the one that corresponds to the part of
  the boundary where the new hexagon was attached) with $s_1 5 s_2$ where
  $s_1 + s_2 = s - 1$.  Then
  \begin{equation*}
    \win(c) = (\suma(c') + 5 - 1) - 2 (\len(c') + 2) = \win(c') = 6.
  \end{equation*}
  Analogous arguments can be used for other types of addition.
\end{proof}

\section{Convex benzenoids and convexity deficit}

\subsection{Graph invariants}
A \emph{graph invariant} is a function from a class of graphs to a class of
values (e.g.\ integers, real numbers, polynomials) that takes the same
value for any two isomorphic graphs. Graph invariants may be categorised by
codomains of the functions that define them. When the codomain is the
Boolean domain, they are called \emph{graph properties}. (For example, a
graph can either be bipartite or non-bipartite.) Numerous \emph{integer
  invariants} exist for graphs: order, size, diameter, girth, genus,
chromatic number, etc. Perhaps the most well-known integer invariant in
chemical graph theory is the Wiener index \cite{Rouvray2002}. An example of
a \emph{real number invariant} is the Estrada index
\cite{PenaGutmanRada2007}. In the literature one can find thousands of
graph invariants.

\subsection{Convex benzenoids}

In 2012, a special sub-family of benzenoids, called \emph{convex
  benzenoids}, was introduced by Cruz, Gutman and
Rada~\cite{CruGutRad12}. This family was further studied and enumerated in
\cite{BasFowPis2017}.  A convex benzenoid can be characterised via its
boundary-edges code.

\begin{definition}
Benzenoid $B$ is \emph{convex} if its boundary-edges code contains no $1$.
\end{definition}
The above statement is Proposition~3 in \cite{BasFowPis2017}. Since this is
one possible characterisation of convex benzenoids we may use it as a
definition here.

We note in passing that for infinite benzenoids the situation is more
complex. As shown in \cite{Bas2017} infinite benzenoids may have more than
one boundary component and may need several infinite codes for its
description. Sometimes the code does not describe an infinite benzenoid up
to isomorphism. An example of such an infinite convex benzenoid that is not
determined by its boundary-edges code is called a \emph{strip} in
\cite{Bas2017}.  Strips of different width have the same boundary-edges
code. Hence, in this paper we focus mainly 
on finite benzenoids.

\subsection{Convexity deficit}

Both convex and non-convex benzenoids play important and sometimes distinct
roles in organic chemistry.  For example, in the simplest case of benzenoid
isomers, convex anthracene comprising three linearly fused hexagons is less
stable than non-convex phenanthrene (see Figure~\ref{fig:convsnon}).
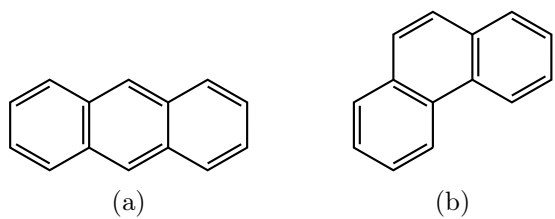
\begin{figure}[!htbp]
\centering
\subfigure[]{
\begin{tikzpicture}[scale=0.6]
\tikzstyle{edge} = [draw,thick]
\coordinate (a1) at (0, 0) {};
\coordinate (a2) at ($ (a1) + (90:1) $) {};
\coordinate (a3) at ($ (a2) + (150:1) $) {};
\coordinate (a4) at ($ (a3) + (210:1) $) {};
\coordinate (a5) at ($ (a4) + (-90:1) $) {};
\coordinate (a6) at ($ (a5) + (-30:1) $) {};
\coordinate (b1) at ($ (a1) + (-30:1) $) {};
\coordinate (b2) at ($ (b1) + (30:1) $) {};
\coordinate (b3) at ($ (b2) + (90:1) $) {};
\coordinate (b4) at ($ (b3) + (150:1) $) {};
\coordinate (c1) at ($ (b2) + (-30:1) $) {};
\coordinate (c2) at ($ (c1) + (30:1) $) {};
\coordinate (c3) at ($ (c2) + (90:1) $) {};
\coordinate (c4) at ($ (c3) + (150:1) $) {};
\path[edge] (a1) -- (a2) -- (a3) -- (a4) -- (a5) -- (a6) -- (a1);
\path[edge] (a1) -- (b1) -- (b2) -- (b3) -- (b4) -- (a2);
\path[edge] (b2) -- (c1) -- (c2) -- (c3) -- (c4) -- (b3);
\path[edge] ($ (a3) + (-90:0.15) $) -- ($ (a4) + (-30:0.15) $);
\path[edge] ($ (a5) + (30:0.15) $) -- ($ (a6) + (90:0.15) $);
\path[edge] ($ (a1) + (150:0.15) $) -- ($ (a2) + (210:0.15) $);
\path[edge] ($ (b1) + (90:0.15) $) -- ($ (b2) + (150:0.15) $);
\path[edge] ($ (b3) + (210:0.15) $) -- ($ (b4) + (-90:0.15) $);
\path[edge] ($ (c1) + (90:0.15) $) -- ($ (c2) + (150:0.15) $);
\path[edge] ($ (c3) + (210:0.15) $) -- ($ (c4) + (-90:0.15) $);
\end{tikzpicture}
}
\qquad
\subfigure[]{
\begin{tikzpicture}[scale=0.6]
\tikzstyle{edge} = [draw,thick]
\coordinate (a1) at (0, 0) {};
\coordinate (a2) at ($ (a1) + (90:1) $) {};
\coordinate (a3) at ($ (a2) + (150:1) $) {};
\coordinate (a4) at ($ (a3) + (210:1) $) {};
\coordinate (a5) at ($ (a4) + (-90:1) $) {};
\coordinate (a6) at ($ (a5) + (-30:1) $) {};
\coordinate (b1) at ($ (a2) + (30:1) $) {};
\coordinate (b2) at ($ (b1) + (90:1) $) {};
\coordinate (b3) at ($ (b2) + (150:1) $) {};
\coordinate (b4) at ($ (b3) + (210:1) $) {};
\coordinate (c1) at ($ (b1) + (-30:1) $) {};
\coordinate (c2) at ($ (c1) + (30:1) $) {};
\coordinate (c3) at ($ (c2) + (90:1) $) {};
\coordinate (c4) at ($ (c3) + (150:1) $) {};
\path[edge] (a1) -- (a2) -- (a3) -- (a4) -- (a5) -- (a6) -- (a1);
\path[edge] (a2) -- (b1) -- (b2) -- (b3) -- (b4) -- (a3);
\path[edge] (b1) -- (c1) -- (c2) -- (c3) -- (c4) -- (b2);
\path[edge] ($ (a4) + (-30:0.15) $) -- ($ (a5) + (30:0.15) $);
\path[edge] ($ (a6) + (90:0.15) $) -- ($ (a1) + (150:0.15) $);
\path[edge] ($ (a2) + (210:0.15) $) -- ($ (a3) + (-90:0.15) $);
\path[edge] ($ (b3) + (-90:0.15) $) -- ($ (b4) + (-30:0.15) $);
\path[edge] ($ (b1) + (30:0.15) $) -- ($ (b2) + (-30:0.15) $);
\path[edge] ($ (c1) + (90:0.15) $) -- ($ (c2) + (150:0.15) $);
\path[edge] ($ (c3) + (210:0.15) $) -- ($ (c4) + (-90:0.15) $);
\end{tikzpicture}
}
\caption{The two $3$-hexagon Kekulean benzenoids: (i) anthracene, (ii) phenanthrene.}
\label{fig:convsnon}
\end{figure}
In qualitative theories, this is variously attributed to the larger number
of Kekul{\'e} structures in phenanthrene (5 vs.\ 4), its higher Fries number
(3 vs.\ 2) or its higher Clar number (2 vs.\ 1), all of which are inextricably
linked to its angular, non-convex shape.  We think it will be useful to
introduce a measure that will tell us by how much the shape of benzenoid
departs from convexity. We call this measure the \emph{convexity deficit}.

\begin{definition}
  A benzenoid $B$ with boundary-edges code $c$ is \emph{$k$-convex},
  $k \geq 0$, if the average of $k+1$ consecutive values in $c$ is always
  at least $2$. The minimum such value of $k$ is called the \emph{convexity
    deficit} of $B$ and is denoted by $\cd(B)$.
\end{definition}
For an infinite benzenoid $B$, we may have $\cd(B) = \infty$.

We may write down a formal definition:
\begin{align*}
  \avg(c, k)      & = \min \left\{ \frac{\suma(d)}{\len(d)} \Bigm|
                    d \subseteq c \text{ and } \len(d) = k \right\} \\
  \cd(B) = \cd(c) &  = \min \left\{ k \mid k \geq 0 \text{ and }
                    \avg(c,k + 1) \geq 2 \right\}
\end{align*}
Note that in the first formula $d \subseteq c$ denotes any subcode
consisting of cyclically consecutive symbols of code $c$.  Note that
convexity deficit generalises the notion of convexity for
benzenoids. Clearly, $\cd(B) = 0$ is equivalent to saying that there is no
$1$ in the code.
\begin{proposition}
A benzenoid $B$ is convex if and only if $\cd(B) = 0$.  \qed
\end{proposition}

\begin{proposition}
\label{thm:kbounded}
If $B$ is a finite benzenoid then
\begin{equation*}
0 \leq \cd(B) \leq \len(\code(B)) - 1.
\end{equation*}
\end{proposition}
\begin{proof}
  Let $c = \code(B)$. Recall the definition of winding. Since
  $\win(c) = \suma(c) - 2\len(c) = 6$, we have
  \begin{equation*}
    \frac{\suma(c)}{\len(c)} = 2 + \frac{6}{\len(c)} > 2.
  \end{equation*}
  Hence any (finite) benzenoid is $k$-convex at least for
  $k = \len(c)-1$.
\end{proof}
We note that for infinite benzenoids there exists no upper bound on the
convexity deficit.

\begin{proposition}
  There exist infinite benzenoids $B$ that are not $k$-convex for any
  finite value of $k \geq 0$.
\end{proposition}
\begin{proof}
  An example is the infinite benzenoid with boundary-edges code
  $\ldots 2221222 \ldots$ shown in Figure~\ref{fig:cdunbounded}.  It is the
  complement of the anvil $\mathcal{AN}$ \cite{Bas2017}.
\end{proof}

\begin{figure}[!htbp]
\centering
\input{figures/limit.tikz}
\caption{The infinite benzenoid with boundary-edges code $\ldots 2221222 \ldots$.}
\label{fig:cdunbounded}
\end{figure}

\subsection{Quasi-convex and pseudo-convex benzenoids}

Now we turn our attention to the non-convex benzenoids that are closest to
convex, i.e.\ the benzenoids with the next smallest convexity deficit.
\begin{definition}
  A $1$-convex benzenoid which is not $0$-convex (i.e., $\cd(B) = 1$) is
  called \emph{quasi-convex}.
\end{definition}

Note that quasi-convex benzenoids admit a simple characterisation via the
boundary-edges code.
\begin{proposition}
  A benzenoid is quasi-convex if and only if its boundary-edges code
  contains at least one $1$ but no sub-sequence $11$, $12$, or $21$.
\end{proposition}
\begin{proof}
  A quasi-convex benzenoid is not convex, hence its code contains a
  $1$. Let $a$ and $b$ be two cyclically consecutive numbers in the code of
  this benzenoid. Since it is $1$-convex, $a + b \geq 4$, hence $11$, $12$,
  and $21$ are forbidden. The converse also follows.
\end{proof}

Convex benzenoids can be classified into families with a common fundamental
shape, 
where zig-zag ($2^k$) sub-sequences define the edges of the shape.
Similarly, all quasi-convex benzenoids have a fundamental shape where the
edges are defined by either zig-zag or armchair ($1(31)^k$) sub-sequences.
The fundamental shape of a convex benzenoid has at most $6$ edges; for a
quasi-convex benzenoid it has at most $12$.  Zig-zag and airmchair
termination have consequences for stability of benzenoids \cite{Dias2005}
and conductivity of nanotubes \cite{Dresselhaus}.  A quasi-convex benzenoid
that has no zig-zag sub-sequences in its boundary-edges code will be called
\emph{pseudo-convex}.
\begin{definition}
  A benzenoid whose boundary-edges code contains at least one $1$ but no
  sub-sequence $11$ and $2$ is called \emph{pseudo-convex}.
\end{definition}

\begin{proposition}
  Every pseudo-convex benzenoid is quasi-convex but the converse is not
  true. \qed
\end{proposition}

Here are some small examples. Note that naphthalene $55$ is convex,
phenanthrene $5351$ is pseudo-convex and benzo[a]pyrene $513432$ is
quasi-convex (but not pseudo-convex). A smaller example of such a benzenoid
is described by $52441$. They are shown in Figures~\ref{fig:theseven} and \ref{fig:thefoursmall}.
The BEC $52441$ applies to the ``pistol'' polyhex \cite{Gardner1978}, named
for its shape.  BEC apply equally to benzenoids and polyhexes.

\begin{figure}[!htbp]
\centering
\subfigure[pistol, $52441$]{
\begin{tikzpicture}[scale=1,xscale=1]
\tikzstyle{every node} = [inner sep=1.2, draw, circle,fill=gray!60]
\tikzstyle{edge} = [draw, line width=1.0]
\tikzstyle{periedge} = [draw, line width=1.0]
\node (h11) at (0.43,0.75) {};
\node (h12) at (0.87,1) {};
\node (h13) at (1.3,0.75) {};
\node (h14) at (1.3,0.25) {};
\node (h15) at (0.87,0) {};
\node (h16) at (0.43,0.25) {};
\node (h21) at (0,1) {};
\node (h22) at (0,1.5) {};
\node (h23) at (0.43,1.75) {};
\node (h24) at (0.87,1.5) {};
\node (h31) at (1.3,1.75) {};
\node (h32) at (1.73,1.5) {};
\node (h33) at (1.73,1) {};
\node (h41) at (2.17,1.75) {};
\node (h42) at (2.6,1.5) {};
\node (h43) at (2.6,1) {};
\node (h44) at (2.17,0.75) {};
\draw[edge] (h11)--(h12)--(h13)--(h14)--(h15)--(h16)--(h11);
\draw[edge] (h11)--(h21)--(h22)--(h23)--(h24)--(h12);
\draw[edge] (h24)--(h31)--(h32)--(h33)--(h13);
\draw[edge] (h32)--(h41)--(h42)--(h43)--(h44)--(h33);
\end{tikzpicture}
}
\subfigure[wave, $513513$]{
\begin{tikzpicture}[scale=1,xscale=1,rotate=30]
\tikzstyle{every node} = [inner sep=1.2, draw, circle,fill=gray!60]
\tikzstyle{edge} = [draw, line width=1.0]
\tikzstyle{periedge} = [draw, line width=1.0]
\node (11) at (0.5,1) {};
\node (12) at (1,1) {};
\node (13) at (1.25,1.43) {};
\node (14) at (1,1.87) {};
\node (15) at (0.5,1.87) {};
\node (16) at (0.25,1.43) {};
\node (21) at (1.25,0.57) {};
\node (22) at (1.75,0.57) {};
\node (23) at (2,1) {};
\node (24) at (1.75,1.43) {};
\node (31) at (2.5,1) {};
\node (32) at (2.75,1.43) {};
\node (33) at (2.5,1.87) {};
\node (34) at (2,1.87) {};
\node (41) at (2.75,0.57) {};
\node (42) at (3.25,0.57) {};
\node (43) at (3.5,1) {};
\node (44) at (3.25,1.43) {};
\draw[edge] (11)--(12)--(13)--(14)--(15)--(16)--(11);
\draw[edge] (12)--(21)--(22)--(23)--(24)--(13);
\draw[edge] (23)--(31)--(32)--(33)--(34)--(24);
\draw[edge] (31)--(41)--(42)--(43)--(44)--(32);
\end{tikzpicture}
}
\subfigure[bee, $4343$]{
\begin{tikzpicture}[scale=1,xscale=1,rotate=30]
\tikzstyle{every node} = [inner sep=1.2, draw, circle,fill=gray!60]
\tikzstyle{edge} = [draw, line width=1.0]
\tikzstyle{periedge} = [draw, line width=1.0]
\node (11) at (0.5,1) {};
\node (12) at (1,1) {};
\node (13) at (1.25,1.43) {};
\node (14) at (1,1.87) {};
\node (15) at (0.5,1.87) {};
\node (16) at (0.25,1.43) {};
\node (21) at (1.25,0.57) {};
\node (22) at (1.75,0.57) {};
\node (23) at (2,1) {};
\node (24) at (1.75,1.43) {};
\node (31) at (2.5,1) {};
\node (32) at (2.75,1.43) {};
\node (33) at (2.5,1.87) {};
\node (34) at (2,1.87) {};
\node (41) at (1.75,2.3) {};
\node (42) at (1.25,2.3) {};
\draw[edge] (11)--(12)--(13)--(14)--(15)--(16)--(11);
\draw[edge] (12)--(21)--(22)--(23)--(24)--(13);
\draw[edge] (23)--(31)--(32)--(33)--(34)--(24);
\draw[edge] (34)--(41)--(42)--(14);
\end{tikzpicture}
}
\subfigure[arch, $533511$]{
\begin{tikzpicture}[scale=1,xscale=1]
\tikzstyle{every node} = [inner sep=1.2, draw, circle,fill=gray!60]
\tikzstyle{edge} = [draw, line width=1.0]
\tikzstyle{periedge} = [draw, line width=1.0]
\node (h11) at (2.6,0) {};
\node (h12) at (3.03,0.25) {};
\node (h13) at (3.03,0.75) {};
\node (h14) at (2.6,1) {};
\node (h15) at (2.17,0.75) {};
\node (h16) at (2.17,0.25) {};
\node (h21) at (2.6,1.5) {};
\node (h22) at (2.17,1.75) {};
\node (h23) at (1.73,1.5) {};
\node (h24) at (1.73,1) {};
\node (h31) at (1.3,1.75) {};
\node (h32) at (0.87,1.5) {};
\node (h33) at (0.87,1) {};
\node (h34) at (1.3,0.75) {};
\node (h41) at (1.3,0.25) {};
\node (h42) at (0.87,0) {};
\node (h43) at (0.44,0.25) {};
\node (h44) at (0.44,0.75) {};
\draw[edge] (h11)--(h12)--(h13)--(h14)--(h15)--(h16)--(h11);
\draw[edge] (h14)--(h21)--(h22)--(h23)--(h24)--(h15);
\draw[edge] (h24)--(h34)--(h33)--(h32)--(h31)--(h23);
\draw[edge] (h34)--(h41)--(h42)--(h43)--(h44)--(h33);
\end{tikzpicture}
}
\\
\subfigure[propeller, $533511$]{
\;\; \begin{tikzpicture}[scale=1,xscale=1]
\tikzstyle{every node} = [inner sep=1.2, draw, circle,fill=gray!60]
\tikzstyle{edge} = [draw, line width=1.0]
\tikzstyle{periedge} = [draw, line width=1.0]
\node (h11) at (2.6,0) {};
\node (h12) at (3.03,0.25) {};
\node (h13) at (3.03,0.75) {};
\node (h14) at (2.6,1) {};
\node (h15) at (2.17,0.75) {};
\node (h16) at (2.17,0.25) {};
\node (h21) at (2.6,1.5) {};
\node (h22) at (2.17,1.75) {};
\node (h23) at (1.73,1.5) {};
\node (h24) at (1.73,1) {};
\node (h31) at (1.3,1.75) {};
\node (h32) at (0.87,1.5) {};
\node (h33) at (0.87,1) {};
\node (h34) at (1.3,0.75) {};
\node (h41) at (2.17,2.25) {};
\node (h42) at (2.6,2.5) {};
\node (h43) at (3.03,2.25) {};
\node (h44) at (3.03,1.75) {};
\draw[edge] (h11)--(h12)--(h13)--(h14)--(h15)--(h16)--(h11);
\draw[edge] (h14)--(h21)--(h22)--(h23)--(h24)--(h15);
\draw[edge] (h24)--(h34)--(h33)--(h32)--(h31)--(h23);
\draw[edge] (h22)--(h41)--(h42)--(h43)--(h44)--(h21);
\end{tikzpicture}\;\;
}
\subfigure[worm, $512523$]{
\begin{tikzpicture}[scale=1,xscale=1]
\tikzstyle{every node} = [inner sep=1.2, draw, circle,fill=gray!60]
\tikzstyle{edge} = [draw, line width=1.0]
\tikzstyle{periedge} = [draw, line width=1.0]
\node (h21) at (0.87,1) {};
\node (h22) at (0.43,0.75) {};
\node (h23) at (0,1) {};
\node (h24) at (0,1.5) {};
\node (h25) at (0.43,1.75) {};
\node (h26) at (0.87,1.5) {};
\node (h31) at (1.3,1.75) {};
\node (h32) at (1.73,1.5) {};
\node (h33) at (1.73,1) {};
\node (h34) at (1.3,0.75) {};
\node (h41) at (2.17,1.75) {};
\node (h42) at (2.6,1.5) {};
\node (h43) at (2.6,1) {};
\node (h44) at (2.17,0.75) {};
\node (h51) at (2.17,2.25) {};
\node (h52) at (2.6,2.5) {};
\node (h53) at (3.04,2.25) {};
\node (h54) at (3.04,1.75) {};
\draw[edge] (h21)--(h22)--(h23)--(h24)--(h25)--(h26)--(h21);
\draw[edge] (h26)--(h31)--(h32)--(h33)--(h34)--(h21);
\draw[edge] (h32)--(h41)--(h42)--(h43)--(h44)--(h33);
\draw[edge] (h41)--(h51)--(h52)--(h53)--(h54)--(h42);
\end{tikzpicture}
}
\subfigure[bar, $522522$]{
\begin{tikzpicture}[scale=1,xscale=1]
\tikzstyle{every node} = [inner sep=1.2, draw, circle,fill=gray!60]
\tikzstyle{edge} = [draw, line width=1.0]
\tikzstyle{periedge} = [draw, line width=1.0]
\node (h21) at (0.87,1) {};
\node (h22) at (0.43,0.75) {};
\node (h23) at (0,1) {};
\node (h24) at (0,1.5) {};
\node (h25) at (0.43,1.75) {};
\node (h26) at (0.87,1.5) {};
\node (h31) at (1.3,1.75) {};
\node (h32) at (1.73,1.5) {};
\node (h33) at (1.73,1) {};
\node (h34) at (1.3,0.75) {};
\node (h41) at (2.17,1.75) {};
\node (h42) at (2.6,1.5) {};
\node (h43) at (2.6,1) {};
\node (h44) at (2.17,0.75) {};
\node (h51) at (3.04,1.75) {};
\node (h52) at (3.47,1.5) {};
\node (h53) at (3.47,1) {};
\node (h54) at (3.04,0.75) {};
\draw[edge] (h21)--(h22)--(h23)--(h24)--(h25)--(h26)--(h21);
\draw[edge] (h26)--(h31)--(h32)--(h33)--(h34)--(h21);
\draw[edge] (h32)--(h41)--(h42)--(h43)--(h44)--(h33);
\draw[edge] (h42)--(h51)--(h52)--(h53)--(h54)--(h43);
\end{tikzpicture}
}
\caption{The seven polyhexes composed of four hexagons \cite{Gardner1978}, and their BEC.}
\label{fig:theseven}
\end{figure}
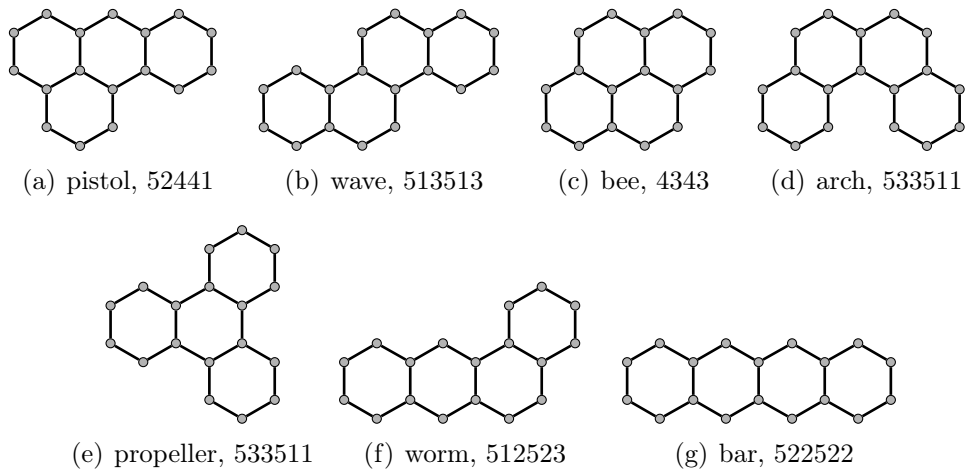

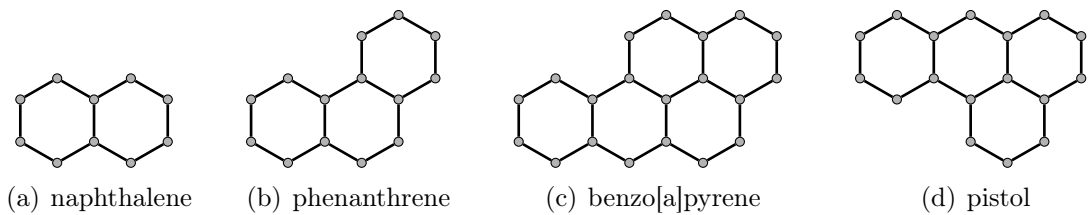
\begin{figure}[!htbp]
\centering
\subfigure[naphthalene]{
\begin{tikzpicture}[scale=0.4]
\tikzstyle{every node} = [inner sep=1.2, draw, circle,fill=gray!60]
\tikzstyle{edge} = [draw, line width=1.0]
\tikzstyle{periedge} = [draw, line width=1.0]
\node[] (-1_0_1) at (-1.212436, 0.700000) {};
\node[] (0_0_0) at (0.000000, 1.400000) {};
\node[] (1_-1_0) at (1.212436, -0.700000) {};
\node[] (-1_-1_1) at (-2.424871, -1.400000) {};
\node[] (-1_0_0) at (-2.424871, 1.400000) {};
\node[] (-1_-1_0) at (-3.637307, -0.700000) {};
\node[] (0_0_1) at (1.212436, 0.700000) {};
\node[] (-2_0_1) at (-3.637307, 0.700000) {};
\node[] (0_-1_1) at (0.000000, -1.400000) {};
\node[] (0_-1_0) at (-1.212436, -0.700000) {};
\draw[periedge] (-1_0_0) -- (-1_0_1);
\draw[periedge] (-1_0_1) -- (0_0_0);
\draw[periedge] (0_0_1) -- (1_-1_0);
\draw[edge] (-1_0_1) -- (0_-1_0);
\draw[periedge] (-1_-1_0) -- (-1_-1_1);
\draw[periedge] (-2_0_1) -- (-1_0_0);
\draw[periedge] (-2_0_1) -- (-1_-1_0);
\draw[periedge] (0_0_0) -- (0_0_1);
\draw[periedge] (0_-1_1) -- (1_-1_0);
\draw[periedge] (0_-1_0) -- (0_-1_1);
\draw[periedge] (-1_-1_1) -- (0_-1_0);
\end{tikzpicture}
}
\quad
\subfigure[phenanthrene]{
\begin{tikzpicture}[scale=-0.4]
\tikzstyle{every node} = [inner sep=1.2, draw, circle,fill=gray!60]
\tikzstyle{edge} = [draw, line width=1.0]
\tikzstyle{periedge} = [draw, line width=1.0]
\node[] (-1_-2_1) at (-3.637307, -3.500000) {};
\node[] (0_0_1) at (1.212436, 0.700000) {};
\node[] (-2_0_1) at (-3.637307, 0.700000) {};
\node[] (1_-1_0) at (1.212436, -0.700000) {};
\node[] (-1_-2_0) at (-4.849742, -2.800000) {};
\node[] (-1_-1_1) at (-2.424871, -1.400000) {};
\node[] (-1_0_1) at (-1.212436, 0.700000) {};
\node[] (0_0_0) at (0.000000, 1.400000) {};
\node[] (0_-2_0) at (-2.424871, -2.800000) {};
\node[] (0_-1_1) at (0.000000, -1.400000) {};
\node[] (-1_0_0) at (-2.424871, 1.400000) {};
\node[] (-1_-1_0) at (-3.637307, -0.700000) {};
\node[] (-2_-1_1) at (-4.849742, -1.400000) {};
\node[] (0_-1_0) at (-1.212436, -0.700000) {};
\draw[periedge] (-1_-1_1) -- (0_-2_0);
\draw[periedge] (-2_-1_1) -- (-1_-2_0);
\draw[periedge] (0_0_1) -- (1_-1_0);
\draw[edge] (-1_0_1) -- (0_-1_0);
\draw[edge] (-1_-1_0) -- (-1_-1_1);
\draw[periedge] (-2_0_1) -- (-1_-1_0);
\draw[periedge] (0_0_0) -- (0_0_1);
\draw[periedge] (0_-1_1) -- (1_-1_0);
\draw[periedge] (-2_-1_1) -- (-1_-1_0);
\draw[periedge] (-1_-2_0) -- (-1_-2_1);
\draw[periedge] (-1_0_0) -- (-1_0_1);
\draw[periedge] (-1_0_1) -- (0_0_0);
\draw[periedge] (-1_-2_1) -- (0_-2_0);
\draw[periedge] (0_-1_0) -- (0_-1_1);
\draw[periedge] (-2_0_1) -- (-1_0_0);
\draw[periedge] (-1_-1_1) -- (0_-1_0);
\end{tikzpicture}
}
\quad
\subfigure[{benzo[a]pyrene}]{
\begin{tikzpicture}[scale=0.4,xscale=-1]
\tikzstyle{every node} = [inner sep=1.2, draw, circle,fill=gray!60]
\tikzstyle{edge} = [draw, line width=1.0]
\tikzstyle{periedge} = [draw, line width=1.0]
\node[] (-2_1_0) at (-3.637307, 3.500000) {};
\node[] (-4_1_1) at (-7.274613, 2.800000) {};
\node[] (-3_0_0) at (-7.274613, 1.400000) {};
\node[] (-2_1_1) at (-2.424871, 2.800000) {};
\node[] (-2_-1_1) at (-4.849742, -1.400000) {};
\node[] (-3_0_1) at (-6.062178, 0.700000) {};
\node[] (0_0_1) at (1.212436, 0.700000) {};
\node[] (-2_0_1) at (-3.637307, 0.700000) {};
\node[] (1_-1_0) at (1.212436, -0.700000) {};
\node[] (-2_-1_0) at (-6.062178, -0.700000) {};
\node[] (-1_-1_1) at (-2.424871, -1.400000) {};
\node[] (-1_0_1) at (-1.212436, 0.700000) {};
\node[] (0_0_0) at (0.000000, 1.400000) {};
\node[] (0_-1_1) at (0.000000, -1.400000) {};
\node[] (-3_1_1) at (-4.849742, 2.800000) {};
\node[] (-1_0_0) at (-2.424871, 1.400000) {};
\node[] (-2_0_0) at (-4.849742, 1.400000) {};
\node[] (-1_-1_0) at (-3.637307, -0.700000) {};
\node[] (-3_1_0) at (-6.062178, 3.500000) {};
\node[] (0_-1_0) at (-1.212436, -0.700000) {};
\draw[periedge] (-4_1_1) -- (-3_0_0);
\draw[periedge] (0_0_1) -- (1_-1_0);
\draw[edge] (-1_0_1) -- (0_-1_0);
\draw[periedge] (-3_0_0) -- (-3_0_1);
\draw[periedge] (0_0_0) -- (0_0_1);
\draw[periedge] (0_-1_0) -- (0_-1_1);
\draw[periedge] (-2_1_1) -- (-1_0_0);
\draw[periedge] (-1_0_1) -- (0_0_0);
\draw[periedge] (-1_-1_1) -- (0_-1_0);
\draw[periedge] (-3_1_0) -- (-3_1_1);
\draw[edge] (-3_0_1) -- (-2_0_0);
\draw[periedge] (-4_1_1) -- (-3_1_0);
\draw[edge] (-2_0_1) -- (-1_-1_0);
\draw[periedge] (-3_1_1) -- (-2_1_0);
\draw[periedge] (-2_1_0) -- (-2_1_1);
\draw[periedge] (-2_-1_0) -- (-2_-1_1);
\draw[edge] (-3_1_1) -- (-2_0_0);
\draw[edge] (-2_0_0) -- (-2_0_1);
\draw[periedge] (0_-1_1) -- (1_-1_0);
\draw[periedge] (-2_-1_1) -- (-1_-1_0);
\draw[periedge] (-1_-1_0) -- (-1_-1_1);
\draw[periedge] (-1_0_0) -- (-1_0_1);
\draw[periedge] (-3_0_1) -- (-2_-1_0);
\draw[edge] (-2_0_1) -- (-1_0_0);
\end{tikzpicture}
}
\quad
\subfigure[pistol]{
\begin{tikzpicture}[scale=0.4,xscale=-1]
\tikzstyle{every node} = [inner sep=1.2, draw, circle,fill=gray!60]
\tikzstyle{edge} = [draw, line width=1.0]
\tikzstyle{periedge} = [draw, line width=1.0]
\node[] (0_-2_0) at (-2.424871, -2.800000) {};
\node[] (-2_-1_1) at (-4.849742, -1.400000) {};
\node[] (-1_0_0) at (-2.424871, 1.400000) {};
\node[] (-3_0_1) at (-6.062178, 0.700000) {};
\node[] (0_0_1) at (1.212436, 0.700000) {};
\node[] (-2_0_1) at (-3.637307, 0.700000) {};
\node[] (1_-1_0) at (1.212436, -0.700000) {};
\node[] (-2_-1_0) at (-6.062178, -0.700000) {};
\node[] (-1_-1_1) at (-2.424871, -1.400000) {};
\node[] (-1_0_1) at (-1.212436, 0.700000) {};
\node[] (0_0_0) at (0.000000, 1.400000) {};
\node[] (0_-1_1) at (0.000000, -1.400000) {};
\node[] (-1_-2_0) at (-4.849742, -2.800000) {};
\node[] (-2_0_0) at (-4.849742, 1.400000) {};
\node[] (-1_-1_0) at (-3.637307, -0.700000) {};
\node[] (-1_-2_1) at (-3.637307, -3.500000) {};
\node[] (0_-1_0) at (-1.212436, -0.700000) {};
\draw[periedge] (-1_-1_1) -- (0_-2_0);
\draw[periedge] (-1_-1_1) -- (0_-1_0);
\draw[periedge] (-2_-1_1) -- (-1_-2_0);
\draw[periedge] (0_0_1) -- (1_-1_0);
\draw[edge] (-1_0_1) -- (0_-1_0);
\draw[periedge] (-2_-1_0) -- (-2_-1_1);
\draw[edge] (-2_0_1) -- (-1_-1_0);
\draw[periedge] (0_0_0) -- (0_0_1);
\draw[periedge] (-1_0_0) -- (-1_0_1);
\draw[periedge] (0_-1_1) -- (1_-1_0);
\draw[edge] (-2_-1_1) -- (-1_-1_0);
\draw[edge] (-1_-1_0) -- (-1_-1_1);
\draw[periedge] (-2_0_0) -- (-2_0_1);
\draw[periedge] (-1_0_1) -- (0_0_0);
\draw[periedge] (0_-1_0) -- (0_-1_1);
\draw[periedge] (-3_0_1) -- (-2_-1_0);
\draw[periedge] (-2_0_1) -- (-1_0_0);
\draw[periedge] (-3_0_1) -- (-2_0_0);
\draw[periedge] (-1_-2_0) -- (-1_-2_1);
\draw[periedge] (-1_-2_1) -- (0_-2_0);
\end{tikzpicture}
}
\caption{Small examples of (a) convex, (b) pseudo-convex and (c), (d)
  quasi-convex benzenoids.}
\label{fig:thefoursmall}
\end{figure}

We have developed software that transforms the boundary edges code to the
description of a benzenoid via position of its hexagons in the hexagonal
tesselation of the plane, as well as a tool that can draw the corresponding
benzenoid. We can also compute several parameters such as convexity deficit
(of course, convexity deficit is obtained directly from the BEC). We present computational
results in Tables \ref{table:small} and \ref{table:family}. 
One is a table of small benzenoids, together with
their names and basic properties. The other lists some of the infinite
families of benzenoids. Some of the representatives of families presented in Table \ref{table:family}
are depicted in Figure \ref{fig:examples}.

\begin{table}[!p]
  \caption{List of small benzenoids and their features. The list is
    complete up to $4$ hexagons.}
\label{table:small}
\centering

 \small
\vspace{0.5\baselineskip}
\setlength{\tabcolsep}{3pt}
\begin{tabular}{| l l r l r | l l |}
  \hline
  \textbf{Name} & \textbf{BEC} & $\bm{h}$ & \textbf{Class} & $\bm{\cd}$ & \textbf{Formula} & \textbf{CAS} \\
  \hline
  \hline
  benzene & 6 & 1 & convex & 0 & $\mathrm{C}_6\mathrm{H}_6$ & 71-43-2 \\
  \hline
  naphthalene & 55 & 2 & convex & 0 & $\mathrm{C}_{10}\mathrm{H}_{8}$ & 91-20-3 \\
  \hline
  phenalene/phenalenyl & 444 & 3 & convex & 0& $\mathrm{C}_{13}\mathrm{H}_{9}$ & 203-80-5 \\
  anthracene & 5252 & 3 & convex & 0& $\mathrm{C}_{14}\mathrm{H}_{10}$ & 120-12-7 \\
  phenanthrene & 5351 & 3 & pseudo-convex & 1& $\mathrm{C}_{14}\mathrm{H}_{10}$ & 85-01-8 \\
  \hline
  tetracene/naphthacene & 522522 & 4 & convex & 0& $\mathrm{C}_{18}\mathrm{H}_{12}$ & 92-24-0 \\
  pyrene & 4343 & 4 & convex & 0& $\mathrm{C}_{16}\mathrm{H}_{10}$ & 129-00-0 \\
  benzophenalenyl & 52441 & 4 & quasi-convex & 1& $\mathrm{C}_{17}\mathrm{H}_{11}$ & 112772-04-0 \\
  chrysene & 513513 & 4 & pseudo-convex & 1& $\mathrm{C}_{18}\mathrm{H}_{12}$ & 218-01-9 \\
  triphenylene & 515151 & 4 & pseudo-convex & 1& $\mathrm{C}_{18}\mathrm{H}_{12}$ & 217-59-4 \\
  benzo(c)phenanthrene & 533511 & 4 &  & 2& $\mathrm{C}_{18}\mathrm{H}_{12}$ & 195-19-7 \\
  benz(a)anthracene & 512523 & 4 &  & 2& $\mathrm{C}_{18}\mathrm{H}_{12}$ & 56-55-3 \\
  \hline
  olympicene/olimpicenyl & 42433 & 5 & convex & 0& $\mathrm{C}_{19}\mathrm{H}_{11}$ & 191-33-3 \\
  pentacene & 52225222 & 5 & convex & 0& $\mathrm{C}_{22}\mathrm{H}_{14}$ & 135-48-8 \\
  picene & 51315313 & 5 & pseudo-convex & 1& $\mathrm{C}_{22}\mathrm{H}_{14}$ & 213-46-7 \\
  {}[$5$]helicene & 53335111 & 5 &  & 3 & $\mathrm{C}_{22}\mathrm{H}_{14}$ & 188-52-3 \\ 
  perylene & 441441 & 5 & pseudo-convex & 1& $\mathrm{C}_{20}\mathrm{H}_{12}$ & 198-55-0 \\
  benzo(a)pyrene & 513432 & 5 & quasi-convex & 1& $\mathrm{C}_{20}\mathrm{H}_{12}$ & 50-32-8 \\
  benzo(e)pyrene & 514341 & 5 & pseudo-convex & 1& $\mathrm{C}_{20}\mathrm{H}_{12}$ & 192-97-2 \\
  dibenz(a,h)anthracene & 53215321 & 5 &  & 2& $\mathrm{C}_{22}\mathrm{H}_{14}$ & 53-70-3 \\
  pentaphene & 52125232 & 5 &  & 3& $\mathrm{C}_{22}\mathrm{H}_{14}$ & 222-93-5 \\
  dibenz(a,j)anthracene & 51215323 & 5 &  & 3& $\mathrm{C}_{22}\mathrm{H}_{14}$ & 224-41-9 \\
  \hline
  triangulenyl & 424242 & 6 & convex & 0& $\mathrm{C}_{22}\mathrm{H}_{12}$ &  \\
  anthanthrene & 324324 & 6 & convex & 0& $\mathrm{C}_{22}\mathrm{H}_{12}$ & 191-26-4 \\
  hexacene & 5222252222 & 6 & convex & 0& $\mathrm{C}_{26}\mathrm{H}_{16}$ & 258-31-1 \\
  benzo(ghi)perylene & 414333 & 6 & pseudo-convex & 1& $\mathrm{C}_{22}\mathrm{H}_{12}$ & 191-24-2 \\
  zethrene & 42144214 & 6 &  &2& $\mathrm{C}_{24}\mathrm{H}_{14}$ & 214-63-1 \\
  \hline
  coronene & 333333 & 7 & convex & 0& $\mathrm{C}_{24}\mathrm{H}_{12}$ & 191-07-1 \\
  heptacene & 522222522222 & 7 & convex & 0& $\mathrm{C}_{30}\mathrm{H}_{18}$ & 258-38-8 \\
  peropyrene  & 43134313 &7& pseudo-convex  & 1 & $\mathrm{C}_{26}\mathrm{H}_{14}$ & 188-96-5 \\
  \hline
  terrylene  & 4413144131 &8& pseudo-convex  & 1& $\mathrm{C}_{30}\mathrm{H}_{16}$ & 188-72-7 \\
  tribenzo[b,n,pqr]perylene & 5141251331 & 8 & & 2 & $\mathrm{C}_{30}\mathrm{H}_{16}$ & 190-81-8 \\ 
  tribenzo[b,k,pqr]perylene  & 5141415131 & 8 &pseudo-convex &1 & $\mathrm{C}_{30}\mathrm{H}_{16}$ &  \\ 
  tribenzo[b,ghi,n]perylene  & 5141251331 & 8 &  & 2 & $\mathrm{C}_{30}\mathrm{H}_{16}$ &  \\ 
  \hline
  ovalene & 33323332 & 10 & convex & 0& $\mathrm{C}_{32}\mathrm{H}_{14}$ & 190-26-1 \\
  teropyrene & 431313431313 & 10& pseudo-convex & 1& $\mathrm{C}_{36}\mathrm{H}_{18}$ &  \\
  \hline
 {\footnotesize hexabenzo[bc,ef,hi,kl,no,qr]coronene} & 414141414141 &13& pseudo-convex & 1& $\mathrm{C}_{42}\mathrm{H}_{18}$ & 190-24-9 \\
  {\footnotesize hexabenzo[a,d,g,j,m,p]coronene} & 511511511511511511 &13&  & 2 & $\mathrm{C}_{48}\mathrm{H}_{24}$ & 1065-80-1 \\
  \hline
  dicoronylene & 23333212333321 & 15 &  & 3& $\mathrm{C}_{48}\mathrm{H}_{20}$ & 98570-53-7 \\
  \hline
\end{tabular}
\end{table}

\begin{table}[p]
\centering
\caption{Some families of benzenoids and their convexity defect.}
\label{table:family}

\small
\renewcommand*{\arraystretch}{1.2}
\vspace{0.5\baselineskip}
\begin{tabular}{|l|l|l|l|l|l|}
\hline
 \multicolumn{2}{|l|}{\textbf{Benzenoid family}} & \textbf{BEC} \\
\hline
 $\bm{h(B)}$ & $\bm{\cd(B)}$ & \textbf{Source} \\
\hline
\hline
Linear & $L(n)$, $n\ge 2$ & $52^{n-2}52^{n-2}$ \\
\hline
$n$ & $0$ (convex) & \cite[p.~62]{cyvin_1988} \\
\hline
\hline
Two segments & $M_2(m,n)$, $m,n>1$ & $52^{m-2}12^{n-2}52^{n-2}32^{m-2}$ \\
\hline
$m + n - 1$ & $m +  n - 3$ & \cite[p.~62]{cyvin_1988} \\
\hline
\hline
Three segments & $M_3(m,n,k)$, $m,n,k>1$ & $52^{k-2}12^{m-2}12^{n-2}52^{n-2}32^{m-2}32^{k-2}$ \\
\hline
$m + n + k - 2 $ & $m + n + k - 4$ & \cite[p.~62]{cyvin_1988} \\
\hline
\hline
Three segments &  $Z_3(m,n,k)$, $m,n,k>1$ & $52^{n-2}12^{k-2}32^{m-2}52^{m-2}12^{k-2}32^{n-2}$ \\
\hline
$m + n + k - 2$ & $\max\{m, n\} + k - 3$ & \cite[p.~62]{cyvin_1988} \\
\hline
\hline
Chevron & $\Ch(n,m,k)$, $n,m,k \geq 2$ & $4 2^{n-2} 3 2^{k-2} 3 2^{m-2} 3 2^{n-2} 4 2^{m-2} 1 2^{k-2}$ \\
\hline
$n(m + k - 1)$ & $m + k - 3$ &\cite{gordon_1952,cyvin_1985}, \cite[p.~111]{cyvin_1988} \\
\hline
\hline
Prolate triangle & $P_3(m)$, $m \geq 2$ & $51(31)^{m-2}52^{m-2}32^{m-2}$ \\
\hline
$\frac{1}{2} m(m + 1)$  & $1$ (quasi-convex) & \cite[p.~182]{cyvin_1988} \\
\hline
\hline
Prolate pentagon & $P_5(m,n)$, $m,n \geq 2$ & $32^{n-2}41(31)^{m-2}42^{n-2}32^{m-2}32^{m-2}$ \\
\hline
$\frac{1}{2} m(m + 1) + (n - 1)(2m - 1)$ & $1$ (quasi-convex) & \cite[p.~182]{cyvin_1988} \\
\hline
\hline
Oblate triangle & $O_3(m)$, $m \geq 2$ & $43(13)^{m-2}42^{m-2}32^{m-2}$ \\
\hline
$\frac{1}{2} m(m + 1) + (m - 1)$ & $\begin{cases}0 \text{ (convex)} & m = 2 \\ 1  \text{ (quasi-convex)}   & m > 2  \end{cases}$ & \cite[p.~197]{cyvin_1988} \\
\hline
\hline
Problate triangle & $B_3(m)$, $m \geq 2$ & $4(31)^{m-1} 52^{m-1}32^{m-2}$ \\
\hline
$\frac{1}{2} m(m + 3)$ & $1$ (quasi-convex) & \cite[p.~197]{cyvin_1988} \\
\hline
\hline
Prolate rectangle &$P_4(m,n)$, $m,n \geq 2$ & $42^{n-2}4(13)^{m-2}142^{n-2}4(13)^{m-2}1$ \\
\hline
$nm + (n - 1)(m - 1)$ & $1$ (quasi-convex) &\cite{yen_1971}, \cite[p.~201]{cyvin_1988} \\
\hline
\hline
Dihedral all-benzenoids & $S(m)$, $m \geq 1$ & $51215(13)^{m-1}151215(13)^{m-1}1$ \\
\hline
$7m$ & 3 & \cite{zhang_1986}, \cite[p.~215]{cyvin_1988} \\
\hline
\hline
 & $T(2)$ & $51415141$ \\
\hline
6 & $1$ (pseudo-convex) & \cite{zhang_1986}, \cite[p.~215]{cyvin_1988} \\
\hline
\hline
 & $T(m)$, $m \geq 3$ & $41414(13)^{m-3}141414(13)^{m-3}1$ \\
\hline
$7m - 8$ & $1$ (pseudo-convex) & \cite{zhang_1986}, \cite[p.~215]{cyvin_1988} \\
\hline
\end{tabular}
\end{table}

\begin{figure}[!htbp]
\centering
\subfigure[$L(7)$]{
\begin{tikzpicture}[scale=0.67,xscale=1]
\tikzstyle{every node} = [inner sep=1.2, draw, circle,fill=gray!60]
\tikzstyle{edge} = [draw, line width=1.0]
\tikzstyle{periedge} = [draw, line width=1.0]
\node (11) at (0.87,1) {};
\node (12) at (1.3,0.75) {};
\node (13) at (1.3,0.25) {};
\node (14) at (0.87,0) {};
\node (15) at (0.43,0.25) {};
\node (16) at (0.43,0.75) {};
\node (21) at (1.73,1) {};
\node (22) at (2.17,0.75) {};
\node (23) at (2.17,0.25) {};
\node (24) at (1.73,0) {};
\node (31) at (2.6,1) {};
\node (32) at (3.03,0.75) {};
\node (33) at (3.03,0.25) {};
\node (34) at (2.6,0) {};
\node (41) at (3.46,1) {};
\node (42) at (3.89,0.75) {};
\node (43) at (3.89,0.25) {};
\node (44) at (3.46,0) {};
\node (51) at (4.32,1) {};
\node (52) at (4.75,0.75) {};
\node (53) at (4.75,0.25) {};
\node (54) at (4.32,0) {};
\node (61) at (5.18,1) {};
\node (62) at (5.61,0.75) {};
\node (63) at (5.61,0.25) {};
\node (64) at (5.18,0) {};
\node (71) at (6.04,1) {};
\node (72) at (6.47,0.75) {};
\node (73) at (6.47,0.25) {};
\node (74) at (6.04,0) {};
\draw[edge] (11)--(12)--(13)--(14)--(15)--(16)--(11);
\draw[edge] (12)--(21)--(22)--(23)--(24)--(13);
\draw[edge] (22)--(31)--(32)--(33)--(34)--(23);
\draw[edge] (32)--(41)--(42)--(43)--(44)--(33);
\draw[edge] (42)--(51)--(52)--(53)--(54)--(43);
\draw[edge] (52)--(61)--(62)--(63)--(64)--(53);
\draw[edge] (62)--(71)--(72)--(73)--(74)--(63);
\end{tikzpicture}
}
\quad
\subfigure[$M_2(4, 3)$]{
\begin{tikzpicture}[scale=0.67,xscale=1]
\tikzstyle{every node} = [inner sep=1.2, draw, circle,fill=gray!60]
\tikzstyle{edge} = [draw, line width=1.0]
\tikzstyle{periedge} = [draw, line width=1.0]
\node (11) at (2.6,1) {};
\node (12) at (3.03,0.75) {};
\node (13) at (3.03,0.25) {};
\node (14) at (2.6,0) {};
\node (15) at (2.17,0.25) {};
\node (16) at (2.17,0.75) {};
\node (21) at (3.46,1) {};
\node (22) at (3.9,0.75) {};
\node (23) at (3.9,0.25) {};
\node (24) at (3.46,0) {};
\node (31) at (4.33,1) {};
\node (32) at (4.76,0.75) {};
\node (33) at (4.76,0.25) {};
\node (34) at (4.33,0) {};
\node (41) at (2.6,1.5) {};
\node (42) at (2.17,1.75) {};
\node (43) at (1.73,1.5) {};
\node (44) at (1.73,1) {};
\node (51) at (2.17,2.25) {};
\node (52) at (1.73,2.5) {};
\node (53) at (1.3,2.25) {};
\node (54) at (1.3,1.75) {};
\node (61) at (1.73,3) {};
\node (62) at (1.3,3.25) {};
\node (63) at (0.87,3) {};
\node (64) at (0.87,2.5) {};
\draw[edge] (11)--(12)--(13)--(14)--(15)--(16)--(11);
\draw[edge] (12)--(21)--(22)--(23)--(24)--(13);
\draw[edge] (22)--(31)--(32)--(33)--(34)--(23);
\draw[edge] (11)--(41)--(42)--(43)--(44)--(16);
\draw[edge] (42)--(51)--(52)--(53)--(54)--(43);
\draw[edge] (52)--(61)--(62)--(63)--(64)--(53);
\end{tikzpicture}
}
\quad
\subfigure[$M_3(4, 3, 5)$]{
\begin{tikzpicture}[scale=0.67,xscale=1]
\tikzstyle{every node} = [inner sep=1.2, draw, circle,fill=gray!60]
\tikzstyle{edge} = [draw, line width=1.0]
\tikzstyle{periedge} = [draw, line width=1.0]
\node (11) at (2.6,1) {};
\node (12) at (3.03,0.75) {};
\node (13) at (3.03,0.25) {};
\node (14) at (2.6,0) {};
\node (15) at (2.17,0.25) {};
\node (16) at (2.17,0.75) {};
\node (21) at (3.46,1) {};
\node (22) at (3.9,0.75) {};
\node (23) at (3.9,0.25) {};
\node (24) at (3.46,0) {};
\node (31) at (4.33,1) {};
\node (32) at (4.76,0.75) {};
\node (33) at (4.76,0.25) {};
\node (34) at (4.33,0) {};
\node (41) at (2.6,1.5) {};
\node (42) at (2.17,1.75) {};
\node (43) at (1.73,1.5) {};
\node (44) at (1.73,1) {};
\node (51) at (2.17,2.25) {};
\node (52) at (1.73,2.5) {};
\node (53) at (1.3,2.25) {};
\node (54) at (1.3,1.75) {};
\node (61) at (5.2,1) {};
\node (62) at (5.63,0.75) {};
\node (63) at (5.63,0.25) {};
\node (64) at (5.2,0) {};
\node (71) at (6.06,1) {};
\node (72) at (6.5,0.75) {};
\node (73) at (6.5,0.25) {};
\node (74) at (6.06,0) {};
\node (81) at (1.73,3) {};
\node (82) at (2.17,3.25) {};
\node (83) at (2.6,3) {};
\node (84) at (2.6,2.5) {};
\node (91) at (2.17,3.75) {};
\node (92) at (2.6,4) {};
\node (93) at (3.03,3.75) {};
\node (94) at (3.03,3.25) {};
\node (01) at (2.6,4.5) {};
\node (02) at (3.03,4.75) {};
\node (03) at (3.46,4.5) {};
\node (04) at (3.46,4) {};
\draw[edge] (11)--(12)--(13)--(14)--(15)--(16)--(11);
\draw[edge] (12)--(21)--(22)--(23)--(24)--(13);
\draw[edge] (22)--(31)--(32)--(33)--(34)--(23);
\draw[edge] (11)--(41)--(42)--(43)--(44)--(16);
\draw[edge] (42)--(51)--(52)--(53)--(54)--(43);
\draw[edge] (32)--(61)--(62)--(63)--(64)--(33);
\draw[edge] (62)--(71)--(72)--(73)--(74)--(63);
\draw[edge] (52)--(81)--(82)--(83)--(84)--(51);
\draw[edge] (82)--(91)--(92)--(93)--(94)--(83);
\draw[edge] (92)--(01)--(02)--(03)--(04)--(93);
\end{tikzpicture}
}
\subfigure[$Z_3(3, 5, 4)$]{
\begin{tikzpicture}[scale=0.67,xscale=1]
\tikzstyle{every node} = [inner sep=1.2, draw, circle,fill=gray!60]
\tikzstyle{edge} = [draw, line width=1.0]
\tikzstyle{periedge} = [draw, line width=1.0]
\node (11) at (2.6,1) {};
\node (12) at (3.03,0.75) {};
\node (13) at (3.03,0.25) {};
\node (14) at (2.6,0) {};
\node (15) at (2.17,0.25) {};
\node (16) at (2.17,0.75) {};
\node (21) at (3.46,1) {};
\node (22) at (3.9,0.75) {};
\node (23) at (3.9,0.25) {};
\node (24) at (3.46,0) {};
\node (31) at (4.33,1) {};
\node (32) at (4.76,0.75) {};
\node (33) at (4.76,0.25) {};
\node (34) at (4.33,0) {};
\node (41) at (2.6,1.5) {};
\node (42) at (2.17,1.75) {};
\node (43) at (1.73,1.5) {};
\node (44) at (1.73,1) {};
\node (51) at (2.17,2.25) {};
\node (52) at (1.73,2.5) {};
\node (53) at (1.3,2.25) {};
\node (54) at (1.3,1.75) {};
\node (61) at (1.73,3) {};
\node (62) at (1.3,3.25) {};
\node (63) at (0.87,3) {};
\node (64) at (0.87,2.5) {};
\node (71) at (1.3,3.75) {};
\node (72) at (0.87,4) {};
\node (73) at (0.43,3.75) {};
\node (74) at (0.43,3.25) {};
\node (81) at (0,4) {};
\node (82) at (-0.43,3.75) {};
\node (83) at (-0.43,3.25) {};
\node (84) at (0,3) {};
\node (91) at (-0.87,4) {};
\node (92) at (-1.3,3.75) {};
\node (93) at (-1.3,3.25) {};
\node (94) at (-0.87,3) {};
\node (01) at (5.2,1) {};
\node (02) at (5.63,0.75) {};
\node (03) at (5.63,0.25) {};
\node (04) at (5.2,0) {};
\draw[edge] (11)--(12)--(13)--(14)--(15)--(16)--(11);
\draw[edge] (12)--(21)--(22)--(23)--(24)--(13);
\draw[edge] (22)--(31)--(32)--(33)--(34)--(23);
\draw[edge] (11)--(41)--(42)--(43)--(44)--(16);
\draw[edge] (42)--(51)--(52)--(53)--(54)--(43);
\draw[edge] (52)--(61)--(62)--(63)--(64)--(53);
\draw[edge] (62)--(71)--(72)--(73)--(74)--(63);
\draw[edge] (73)--(81)--(82)--(83)--(84)--(74);
\draw[edge] (82)--(91)--(92)--(93)--(94)--(83);
\draw[edge] (32)--(01)--(02)--(03)--(04)--(33);
\end{tikzpicture}
}
\quad
\subfigure[$\Ch(3, 4, 2)$]{
\begin{tikzpicture}[scale=0.238,rotate=60]
\tikzstyle{every node} = [inner sep=1.2, draw, circle,fill=gray!60]
\tikzstyle{edge} = [draw, line width=1.0]
\tikzstyle{periedge} = [draw, line width=1.0]
\node (0_0_0) at (0.000000, 1.400000) {};
\node (0_0_1) at (1.212436, 0.700000) {};
\node (1_-1_0) at (1.212436, -0.700000) {};
\node (0_-1_1) at (0.000000, -1.400000) {};
\node (0_-1_0) at (-1.212436, -0.700000) {};
\node (-1_0_1) at (-1.212436, 0.700000) {};
\node (-1_-3_0) at (-6.062178, -4.900000) {};
\node (-1_-3_1) at (-4.849742, -5.600000) {};
\node (0_-4_0) at (-4.849742, -7.000000) {};
\node (-1_-4_1) at (-6.062178, -7.700000) {};
\node (-1_-4_0) at (-7.274613, -7.000000) {};
\node (-2_-3_1) at (-7.274613, -5.600000) {};
\node (-1_-2_0) at (-4.849742, -2.800000) {};
\node (-1_-2_1) at (-3.637307, -3.500000) {};
\node (0_-3_0) at (-3.637307, -4.900000) {};
\node (-2_-2_1) at (-6.062178, -3.500000) {};
\node (-1_-1_0) at (-3.637307, -0.700000) {};
\node (-1_-1_1) at (-2.424871, -1.400000) {};
\node (0_-2_0) at (-2.424871, -2.800000) {};
\node (-2_-1_1) at (-4.849742, -1.400000) {};
\node (-1_0_0) at (-2.424871, 1.400000) {};
\node (-2_0_1) at (-3.637307, 0.700000) {};
\node (-1_1_0) at (-1.212436, 3.500000) {};
\node (-1_1_1) at (0.000000, 2.800000) {};
\node (-2_1_1) at (-2.424871, 2.800000) {};
\node (-2_-2_0) at (-7.274613, -2.800000) {};
\node (-2_-3_0) at (-8.487049, -4.900000) {};
\node (-3_-2_1) at (-8.487049, -3.500000) {};
\node (-2_-1_0) at (-6.062178, -0.700000) {};
\node (-3_-1_1) at (-7.274613, -1.400000) {};
\node (-2_0_0) at (-4.849742, 1.400000) {};
\node (-3_0_1) at (-6.062178, 0.700000) {};
\node (-2_1_0) at (-3.637307, 3.500000) {};
\node (-3_1_1) at (-4.849742, 2.800000) {};
\node (-2_2_0) at (-2.424871, 5.600000) {};
\node (-2_2_1) at (-1.212436, 4.900000) {};
\node (-3_2_1) at (-3.637307, 4.900000) {};
\node (-3_-1_0) at (-8.487049, -0.700000) {};
\node (-3_-2_0) at (-9.699485, -2.800000) {};
\node (-4_-1_1) at (-9.699485, -1.400000) {};
\node (-3_0_0) at (-7.274613, 1.400000) {};
\node (-4_0_1) at (-8.487049, 0.700000) {};
\node (-3_1_0) at (-6.062178, 3.500000) {};
\node (-4_1_1) at (-7.274613, 2.800000) {};
\node (-3_2_0) at (-4.849742, 5.600000) {};
\node (-4_2_1) at (-6.062178, 4.900000) {};
\draw[edge] (-1_0_1) -- (0_0_0);
\draw[periedge] (0_0_0) -- (0_0_1);
\draw[periedge] (0_0_1) -- (1_-1_0);
\draw[periedge] (0_-1_1) -- (1_-1_0);
\draw[periedge] (0_-1_0) -- (0_-1_1);
\draw[edge] (-1_0_1) -- (0_-1_0);
\draw[edge] (-2_-3_1) -- (-1_-3_0);
\draw[edge] (-1_-3_0) -- (-1_-3_1);
\draw[periedge] (-1_-3_1) -- (0_-4_0);
\draw[periedge] (-1_-4_1) -- (0_-4_0);
\draw[periedge] (-1_-4_0) -- (-1_-4_1);
\draw[periedge] (-2_-3_1) -- (-1_-4_0);
\draw[edge] (-2_-2_1) -- (-1_-2_0);
\draw[edge] (-1_-2_0) -- (-1_-2_1);
\draw[periedge] (-1_-2_1) -- (0_-3_0);
\draw[periedge] (-1_-3_1) -- (0_-3_0);
\draw[edge] (-2_-2_1) -- (-1_-3_0);
\draw[edge] (-2_-1_1) -- (-1_-1_0);
\draw[edge] (-1_-1_0) -- (-1_-1_1);
\draw[periedge] (-1_-1_1) -- (0_-2_0);
\draw[periedge] (-1_-2_1) -- (0_-2_0);
\draw[edge] (-2_-1_1) -- (-1_-2_0);
\draw[edge] (-2_0_1) -- (-1_0_0);
\draw[edge] (-1_0_0) -- (-1_0_1);
\draw[periedge] (-1_-1_1) -- (0_-1_0);
\draw[edge] (-2_0_1) -- (-1_-1_0);
\draw[edge] (-2_1_1) -- (-1_1_0);
\draw[periedge] (-1_1_0) -- (-1_1_1);
\draw[periedge] (-1_1_1) -- (0_0_0);
\draw[edge] (-2_1_1) -- (-1_0_0);
\draw[edge] (-3_-2_1) -- (-2_-2_0);
\draw[edge] (-2_-2_0) -- (-2_-2_1);
\draw[periedge] (-2_-3_0) -- (-2_-3_1);
\draw[periedge] (-3_-2_1) -- (-2_-3_0);
\draw[edge] (-3_-1_1) -- (-2_-1_0);
\draw[edge] (-2_-1_0) -- (-2_-1_1);
\draw[edge] (-3_-1_1) -- (-2_-2_0);
\draw[edge] (-3_0_1) -- (-2_0_0);
\draw[edge] (-2_0_0) -- (-2_0_1);
\draw[edge] (-3_0_1) -- (-2_-1_0);
\draw[edge] (-3_1_1) -- (-2_1_0);
\draw[edge] (-2_1_0) -- (-2_1_1);
\draw[edge] (-3_1_1) -- (-2_0_0);
\draw[periedge] (-3_2_1) -- (-2_2_0);
\draw[periedge] (-2_2_0) -- (-2_2_1);
\draw[periedge] (-2_2_1) -- (-1_1_0);
\draw[edge] (-3_2_1) -- (-2_1_0);
\draw[periedge] (-4_-1_1) -- (-3_-1_0);
\draw[edge] (-3_-1_0) -- (-3_-1_1);
\draw[periedge] (-3_-2_0) -- (-3_-2_1);
\draw[periedge] (-4_-1_1) -- (-3_-2_0);
\draw[periedge] (-4_0_1) -- (-3_0_0);
\draw[edge] (-3_0_0) -- (-3_0_1);
\draw[periedge] (-4_0_1) -- (-3_-1_0);
\draw[periedge] (-4_1_1) -- (-3_1_0);
\draw[edge] (-3_1_0) -- (-3_1_1);
\draw[periedge] (-4_1_1) -- (-3_0_0);
\draw[periedge] (-4_2_1) -- (-3_2_0);
\draw[periedge] (-3_2_0) -- (-3_2_1);
\draw[periedge] (-4_2_1) -- (-3_1_0);
\end{tikzpicture}
}
\quad
\subfigure[$P_3(3)$]{
\begin{tikzpicture}[scale=0.67,xscale=1,rotate=-60]
\tikzstyle{every node} = [inner sep=1.2, draw, circle,fill=gray!60]
\tikzstyle{edge} = [draw, line width=1.0]
\tikzstyle{periedge} = [draw, line width=1.0]
\node (11) at (2.17,0.25) {};
\node (12) at (2.17,0.75) {};
\node (13) at (2.6,1) {};
\node (14) at (3.03,0.75) {};
\node (15) at (3.03,0.25) {};
\node (16) at (2.6,0) {};
\node (21) at (3.46,1) {};
\node (22) at (3.9,0.75) {};
\node (23) at (3.9,0.25) {};
\node (24) at (3.46,0) {};
\node (31) at (4.33,1) {};
\node (32) at (4.76,0.75) {};
\node (33) at (4.76,0.25) {};
\node (34) at (4.33,0) {};
\node (41) at (1.73,1) {};
\node (42) at (1.73,1.5) {};
\node (43) at (2.17,1.75) {};
\node (44) at (2.6,1.5) {};
\node (51) at (3.03,1.75) {};
\node (52) at (3.46,1.5) {};
\node (61) at (1.3,1.75) {};
\node (62) at (1.3,2.25) {};
\node (63) at (1.73,2.5) {};
\node (64) at (2.17,2.25) {};
\draw[edge] (11)--(12)--(13)--(14)--(15)--(16)--(11);
\draw[edge] (14)--(21)--(22)--(23)--(24)--(15);
\draw[edge] (22)--(31)--(32)--(33)--(34)--(23);
\draw[edge] (12)--(41)--(42)--(43)--(44)--(13);
\draw[edge] (44)--(51)--(52)--(21);
\draw[edge] (42)--(61)--(62)--(63)--(64)--(43);
\end{tikzpicture}
}
\quad
\subfigure[$P_5(3, 4)$]{
\begin{tikzpicture}[scale=0.238,rotate=-60]
\tikzstyle{every node} = [inner sep=1.2, draw, circle,fill=gray!60]
\tikzstyle{edge} = [draw, line width=1.0]
\tikzstyle{periedge} = [draw, line width=1.0]
\node (0_0_0) at (0.000000, 1.400000) {};
\node (0_0_1) at (1.212436, 0.700000) {};
\node (1_-1_0) at (1.212436, -0.700000) {};
\node (0_-1_1) at (0.000000, -1.400000) {};
\node (0_-1_0) at (-1.212436, -0.700000) {};
\node (-1_0_1) at (-1.212436, 0.700000) {};
\node (0_1_0) at (1.212436, 3.500000) {};
\node (0_1_1) at (2.424871, 2.800000) {};
\node (1_0_0) at (2.424871, 1.400000) {};
\node (-1_1_1) at (0.000000, 2.800000) {};
\node (0_2_0) at (2.424871, 5.600000) {};
\node (0_2_1) at (3.637307, 4.900000) {};
\node (1_1_0) at (3.637307, 3.500000) {};
\node (-1_2_1) at (1.212436, 4.900000) {};
\node (0_3_0) at (3.637307, 7.700000) {};
\node (0_3_1) at (4.849742, 7.000000) {};
\node (1_2_0) at (4.849742, 5.600000) {};
\node (-1_3_1) at (2.424871, 7.000000) {};
\node (-1_0_0) at (-2.424871, 1.400000) {};
\node (-1_-1_1) at (-2.424871, -1.400000) {};
\node (-1_-1_0) at (-3.637307, -0.700000) {};
\node (-2_0_1) at (-3.637307, 0.700000) {};
\node (-1_1_0) at (-1.212436, 3.500000) {};
\node (-2_1_1) at (-2.424871, 2.800000) {};
\node (-1_2_0) at (0.000000, 5.600000) {};
\node (-2_2_1) at (-1.212436, 4.900000) {};
\node (-1_3_0) at (1.212436, 7.700000) {};
\node (-2_3_1) at (0.000000, 7.000000) {};
\node (-2_0_0) at (-4.849742, 1.400000) {};
\node (-2_-1_1) at (-4.849742, -1.400000) {};
\node (-2_-1_0) at (-6.062178, -0.700000) {};
\node (-3_0_1) at (-6.062178, 0.700000) {};
\node (-2_1_0) at (-3.637307, 3.500000) {};
\node (-3_1_1) at (-4.849742, 2.800000) {};
\node (-2_2_0) at (-2.424871, 5.600000) {};
\node (-3_2_1) at (-3.637307, 4.900000) {};
\node (-2_3_0) at (-1.212436, 7.700000) {};
\node (-3_3_1) at (-2.424871, 7.000000) {};
\node (-2_4_0) at (0.000000, 9.800000) {};
\node (-2_4_1) at (1.212436, 9.100000) {};
\node (-3_4_1) at (-1.212436, 9.100000) {};
\node (-3_1_0) at (-6.062178, 3.500000) {};
\node (-3_0_0) at (-7.274613, 1.400000) {};
\node (-4_1_1) at (-7.274613, 2.800000) {};
\node (-3_2_0) at (-4.849742, 5.600000) {};
\node (-4_2_1) at (-6.062178, 4.900000) {};
\node (-3_3_0) at (-3.637307, 7.700000) {};
\node (-4_3_1) at (-4.849742, 7.000000) {};
\node (-3_4_0) at (-2.424871, 9.800000) {};
\node (-4_4_1) at (-3.637307, 9.100000) {};
\node (-4_2_0) at (-7.274613, 5.600000) {};
\node (-4_1_0) at (-8.487049, 3.500000) {};
\node (-5_2_1) at (-8.487049, 4.900000) {};
\node (-4_3_0) at (-6.062178, 7.700000) {};
\node (-5_3_1) at (-7.274613, 7.000000) {};
\node (-4_4_0) at (-4.849742, 9.800000) {};
\node (-5_4_1) at (-6.062178, 9.100000) {};
\node (-4_5_0) at (-3.637307, 11.900000) {};
\node (-4_5_1) at (-2.424871, 11.200000) {};
\node (-5_5_1) at (-4.849742, 11.200000) {};
\draw[edge] (-1_0_1) -- (0_0_0);
\draw[edge] (0_0_0) -- (0_0_1);
\draw[periedge] (0_0_1) -- (1_-1_0);
\draw[periedge] (0_-1_1) -- (1_-1_0);
\draw[periedge] (0_-1_0) -- (0_-1_1);
\draw[edge] (-1_0_1) -- (0_-1_0);
\draw[edge] (-1_1_1) -- (0_1_0);
\draw[edge] (0_1_0) -- (0_1_1);
\draw[periedge] (0_1_1) -- (1_0_0);
\draw[periedge] (0_0_1) -- (1_0_0);
\draw[edge] (-1_1_1) -- (0_0_0);
\draw[edge] (-1_2_1) -- (0_2_0);
\draw[edge] (0_2_0) -- (0_2_1);
\draw[periedge] (0_2_1) -- (1_1_0);
\draw[periedge] (0_1_1) -- (1_1_0);
\draw[edge] (-1_2_1) -- (0_1_0);
\draw[periedge] (-1_3_1) -- (0_3_0);
\draw[periedge] (0_3_0) -- (0_3_1);
\draw[periedge] (0_3_1) -- (1_2_0);
\draw[periedge] (0_2_1) -- (1_2_0);
\draw[edge] (-1_3_1) -- (0_2_0);
\draw[edge] (-2_0_1) -- (-1_0_0);
\draw[edge] (-1_0_0) -- (-1_0_1);
\draw[periedge] (-1_-1_1) -- (0_-1_0);
\draw[periedge] (-1_-1_0) -- (-1_-1_1);
\draw[edge] (-2_0_1) -- (-1_-1_0);
\draw[edge] (-2_1_1) -- (-1_1_0);
\draw[edge] (-1_1_0) -- (-1_1_1);
\draw[edge] (-2_1_1) -- (-1_0_0);
\draw[edge] (-2_2_1) -- (-1_2_0);
\draw[edge] (-1_2_0) -- (-1_2_1);
\draw[edge] (-2_2_1) -- (-1_1_0);
\draw[edge] (-2_3_1) -- (-1_3_0);
\draw[periedge] (-1_3_0) -- (-1_3_1);
\draw[edge] (-2_3_1) -- (-1_2_0);
\draw[edge] (-3_0_1) -- (-2_0_0);
\draw[edge] (-2_0_0) -- (-2_0_1);
\draw[periedge] (-2_-1_1) -- (-1_-1_0);
\draw[periedge] (-2_-1_0) -- (-2_-1_1);
\draw[periedge] (-3_0_1) -- (-2_-1_0);
\draw[edge] (-3_1_1) -- (-2_1_0);
\draw[edge] (-2_1_0) -- (-2_1_1);
\draw[edge] (-3_1_1) -- (-2_0_0);
\draw[edge] (-3_2_1) -- (-2_2_0);
\draw[edge] (-2_2_0) -- (-2_2_1);
\draw[edge] (-3_2_1) -- (-2_1_0);
\draw[edge] (-3_3_1) -- (-2_3_0);
\draw[edge] (-2_3_0) -- (-2_3_1);
\draw[edge] (-3_3_1) -- (-2_2_0);
\draw[periedge] (-3_4_1) -- (-2_4_0);
\draw[periedge] (-2_4_0) -- (-2_4_1);
\draw[periedge] (-2_4_1) -- (-1_3_0);
\draw[edge] (-3_4_1) -- (-2_3_0);
\draw[edge] (-4_1_1) -- (-3_1_0);
\draw[edge] (-3_1_0) -- (-3_1_1);
\draw[periedge] (-3_0_0) -- (-3_0_1);
\draw[periedge] (-4_1_1) -- (-3_0_0);
\draw[edge] (-4_2_1) -- (-3_2_0);
\draw[edge] (-3_2_0) -- (-3_2_1);
\draw[edge] (-4_2_1) -- (-3_1_0);
\draw[edge] (-4_3_1) -- (-3_3_0);
\draw[edge] (-3_3_0) -- (-3_3_1);
\draw[edge] (-4_3_1) -- (-3_2_0);
\draw[edge] (-4_4_1) -- (-3_4_0);
\draw[periedge] (-3_4_0) -- (-3_4_1);
\draw[edge] (-4_4_1) -- (-3_3_0);
\draw[periedge] (-5_2_1) -- (-4_2_0);
\draw[edge] (-4_2_0) -- (-4_2_1);
\draw[periedge] (-4_1_0) -- (-4_1_1);
\draw[periedge] (-5_2_1) -- (-4_1_0);
\draw[periedge] (-5_3_1) -- (-4_3_0);
\draw[edge] (-4_3_0) -- (-4_3_1);
\draw[periedge] (-5_3_1) -- (-4_2_0);
\draw[periedge] (-5_4_1) -- (-4_4_0);
\draw[edge] (-4_4_0) -- (-4_4_1);
\draw[periedge] (-5_4_1) -- (-4_3_0);
\draw[periedge] (-5_5_1) -- (-4_5_0);
\draw[periedge] (-4_5_0) -- (-4_5_1);
\draw[periedge] (-4_5_1) -- (-3_4_0);
\draw[periedge] (-5_5_1) -- (-4_4_0);
\end{tikzpicture}
}
\subfigure[$B_3(3)$]{
\begin{tikzpicture}[scale=0.67,xscale=1,rotate=-60]
\tikzstyle{every node} = [inner sep=1.2, draw, circle,fill=gray!60]
\tikzstyle{edge} = [draw, line width=1.0]
\tikzstyle{periedge} = [draw, line width=1.0]
\node (11) at (2.17,0.25) {};
\node (12) at (2.17,0.75) {};
\node (13) at (2.6,1) {};
\node (14) at (3.03,0.75) {};
\node (15) at (3.03,0.25) {};
\node (16) at (2.6,0) {};
\node (21) at (3.46,1) {};
\node (22) at (3.9,0.75) {};
\node (23) at (3.9,0.25) {};
\node (24) at (3.46,0) {};
\node (31) at (4.33,1) {};
\node (32) at (4.76,0.75) {};
\node (33) at (4.76,0.25) {};
\node (34) at (4.33,0) {};
\node (41) at (1.73,1) {};
\node (42) at (1.73,1.5) {};
\node (43) at (2.17,1.75) {};
\node (44) at (2.6,1.5) {};
\node (51) at (3.03,1.75) {};
\node (52) at (3.46,1.5) {};
\node (61) at (1.3,1.75) {};
\node (62) at (1.3,2.25) {};
\node (63) at (1.73,2.5) {};
\node (64) at (2.17,2.25) {};
\node (71) at (3.9,1.75) {};
\node (72) at (4.33,1.5) {};
\node (81) at (2.6,2.5) {};
\node (82) at (3.03,2.25) {};
\node (91) at (0.87,2.5) {};
\node (92) at (0.87,3) {};
\node (93) at (1.3,3.25) {};
\node (94) at (1.73,3) {};
\draw[edge] (11)--(12)--(13)--(14)--(15)--(16)--(11);
\draw[edge] (14)--(21)--(22)--(23)--(24)--(15);
\draw[edge] (22)--(31)--(32)--(33)--(34)--(23);
\draw[edge] (12)--(41)--(42)--(43)--(44)--(13);
\draw[edge] (44)--(51)--(52)--(21);
\draw[edge] (42)--(61)--(62)--(63)--(64)--(43);
\draw[edge] (52)--(71)--(72)--(31);
\draw[edge] (64)--(81)--(82)--(51);
\draw[edge] (62)--(91)--(92)--(93)--(94)--(63);
\end{tikzpicture}
}
\quad
\subfigure[$O_3(3)$]{
\begin{tikzpicture}[scale=0.67,xscale=1,rotate=-60]
\tikzstyle{every node} = [inner sep=1.2, draw, circle,fill=gray!60]
\tikzstyle{edge} = [draw, line width=1.0]
\tikzstyle{periedge} = [draw, line width=1.0]
\node (11) at (2.17,0.25) {};
\node (12) at (2.17,0.75) {};
\node (13) at (2.6,1) {};
\node (14) at (3.03,0.75) {};
\node (15) at (3.03,0.25) {};
\node (16) at (2.6,0) {};
\node (21) at (3.46,1) {};
\node (22) at (3.9,0.75) {};
\node (23) at (3.9,0.25) {};
\node (24) at (3.46,0) {};
\node (31) at (4.33,1) {};
\node (32) at (4.76,0.75) {};
\node (33) at (4.76,0.25) {};
\node (34) at (4.33,0) {};
\node (41) at (1.73,1) {};
\node (42) at (1.73,1.5) {};
\node (43) at (2.17,1.75) {};
\node (44) at (2.6,1.5) {};
\node (51) at (3.03,1.75) {};
\node (52) at (3.46,1.5) {};
\node (61) at (1.3,1.75) {};
\node (62) at (1.3,2.25) {};
\node (63) at (1.73,2.5) {};
\node (64) at (2.17,2.25) {};
\node (71) at (3.9,1.75) {};
\node (72) at (4.33,1.5) {};
\node (81) at (2.6,2.5) {};
\node (82) at (3.03,2.25) {};
\draw[edge] (11)--(12)--(13)--(14)--(15)--(16)--(11);
\draw[edge] (14)--(21)--(22)--(23)--(24)--(15);
\draw[edge] (22)--(31)--(32)--(33)--(34)--(23);
\draw[edge] (12)--(41)--(42)--(43)--(44)--(13);
\draw[edge] (44)--(51)--(52)--(21);
\draw[edge] (42)--(61)--(62)--(63)--(64)--(43);
\draw[edge] (52)--(71)--(72)--(31);
\draw[edge] (64)--(81)--(82)--(51);
\end{tikzpicture}
}
\quad
\subfigure[$P_4(3, 4)$]{
\begin{tikzpicture}[scale=0.238,rotate=60]
\tikzstyle{every node} = [inner sep=1.2, draw, circle,fill=gray!60]
\tikzstyle{edge} = [draw, line width=1.0]
\tikzstyle{periedge} = [draw, line width=1.0]
\node (0_0_0) at (0.000000, 1.400000) {};
\node (0_0_1) at (1.212436, 0.700000) {};
\node (1_-1_0) at (1.212436, -0.700000) {};
\node (0_-1_1) at (0.000000, -1.400000) {};
\node (0_-1_0) at (-1.212436, -0.700000) {};
\node (-1_0_1) at (-1.212436, 0.700000) {};
\node (-1_-1_0) at (-3.637307, -0.700000) {};
\node (-1_-1_1) at (-2.424871, -1.400000) {};
\node (0_-2_0) at (-2.424871, -2.800000) {};
\node (-1_-2_1) at (-3.637307, -3.500000) {};
\node (-1_-2_0) at (-4.849742, -2.800000) {};
\node (-2_-1_1) at (-4.849742, -1.400000) {};
\node (-1_0_0) at (-2.424871, 1.400000) {};
\node (-2_0_1) at (-3.637307, 0.700000) {};
\node (-1_1_0) at (-1.212436, 3.500000) {};
\node (-1_1_1) at (0.000000, 2.800000) {};
\node (-2_1_1) at (-2.424871, 2.800000) {};
\node (-2_-2_0) at (-7.274613, -2.800000) {};
\node (-2_-2_1) at (-6.062178, -3.500000) {};
\node (-1_-3_0) at (-6.062178, -4.900000) {};
\node (-2_-3_1) at (-7.274613, -5.600000) {};
\node (-2_-3_0) at (-8.487049, -4.900000) {};
\node (-3_-2_1) at (-8.487049, -3.500000) {};
\node (-2_-1_0) at (-6.062178, -0.700000) {};
\node (-3_-1_1) at (-7.274613, -1.400000) {};
\node (-2_0_0) at (-4.849742, 1.400000) {};
\node (-3_0_1) at (-6.062178, 0.700000) {};
\node (-2_1_0) at (-3.637307, 3.500000) {};
\node (-3_1_1) at (-4.849742, 2.800000) {};
\node (-2_2_0) at (-2.424871, 5.600000) {};
\node (-2_2_1) at (-1.212436, 4.900000) {};
\node (-3_2_1) at (-3.637307, 4.900000) {};
\node (-3_-1_0) at (-8.487049, -0.700000) {};
\node (-3_-2_0) at (-9.699485, -2.800000) {};
\node (-4_-1_1) at (-9.699485, -1.400000) {};
\node (-3_0_0) at (-7.274613, 1.400000) {};
\node (-4_0_1) at (-8.487049, 0.700000) {};
\node (-3_1_0) at (-6.062178, 3.500000) {};
\node (-4_1_1) at (-7.274613, 2.800000) {};
\node (-3_2_0) at (-4.849742, 5.600000) {};
\node (-4_2_1) at (-6.062178, 4.900000) {};
\node (-3_3_0) at (-3.637307, 7.700000) {};
\node (-3_3_1) at (-2.424871, 7.000000) {};
\node (-4_3_1) at (-4.849742, 7.000000) {};
\node (-4_0_0) at (-9.699485, 1.400000) {};
\node (-4_-1_0) at (-10.911920, -0.700000) {};
\node (-5_0_1) at (-10.911920, 0.700000) {};
\node (-4_1_0) at (-8.487049, 3.500000) {};
\node (-5_1_1) at (-9.699485, 2.800000) {};
\node (-4_2_0) at (-7.274613, 5.600000) {};
\node (-5_2_1) at (-8.487049, 4.900000) {};
\node (-5_1_0) at (-10.911920, 3.500000) {};
\node (-5_0_0) at (-12.124356, 1.400000) {};
\node (-6_1_1) at (-12.124356, 2.800000) {};
\draw[edge] (-1_0_1) -- (0_0_0);
\draw[periedge] (0_0_0) -- (0_0_1);
\draw[periedge] (0_0_1) -- (1_-1_0);
\draw[periedge] (0_-1_1) -- (1_-1_0);
\draw[periedge] (0_-1_0) -- (0_-1_1);
\draw[edge] (-1_0_1) -- (0_-1_0);
\draw[edge] (-2_-1_1) -- (-1_-1_0);
\draw[edge] (-1_-1_0) -- (-1_-1_1);
\draw[periedge] (-1_-1_1) -- (0_-2_0);
\draw[periedge] (-1_-2_1) -- (0_-2_0);
\draw[periedge] (-1_-2_0) -- (-1_-2_1);
\draw[edge] (-2_-1_1) -- (-1_-2_0);
\draw[edge] (-2_0_1) -- (-1_0_0);
\draw[edge] (-1_0_0) -- (-1_0_1);
\draw[periedge] (-1_-1_1) -- (0_-1_0);
\draw[edge] (-2_0_1) -- (-1_-1_0);
\draw[edge] (-2_1_1) -- (-1_1_0);
\draw[periedge] (-1_1_0) -- (-1_1_1);
\draw[periedge] (-1_1_1) -- (0_0_0);
\draw[edge] (-2_1_1) -- (-1_0_0);
\draw[edge] (-3_-2_1) -- (-2_-2_0);
\draw[edge] (-2_-2_0) -- (-2_-2_1);
\draw[periedge] (-2_-2_1) -- (-1_-3_0);
\draw[periedge] (-2_-3_1) -- (-1_-3_0);
\draw[periedge] (-2_-3_0) -- (-2_-3_1);
\draw[periedge] (-3_-2_1) -- (-2_-3_0);
\draw[edge] (-3_-1_1) -- (-2_-1_0);
\draw[edge] (-2_-1_0) -- (-2_-1_1);
\draw[periedge] (-2_-2_1) -- (-1_-2_0);
\draw[edge] (-3_-1_1) -- (-2_-2_0);
\draw[edge] (-3_0_1) -- (-2_0_0);
\draw[edge] (-2_0_0) -- (-2_0_1);
\draw[edge] (-3_0_1) -- (-2_-1_0);
\draw[edge] (-3_1_1) -- (-2_1_0);
\draw[edge] (-2_1_0) -- (-2_1_1);
\draw[edge] (-3_1_1) -- (-2_0_0);
\draw[edge] (-3_2_1) -- (-2_2_0);
\draw[periedge] (-2_2_0) -- (-2_2_1);
\draw[periedge] (-2_2_1) -- (-1_1_0);
\draw[edge] (-3_2_1) -- (-2_1_0);
\draw[edge] (-4_-1_1) -- (-3_-1_0);
\draw[edge] (-3_-1_0) -- (-3_-1_1);
\draw[periedge] (-3_-2_0) -- (-3_-2_1);
\draw[periedge] (-4_-1_1) -- (-3_-2_0);
\draw[edge] (-4_0_1) -- (-3_0_0);
\draw[edge] (-3_0_0) -- (-3_0_1);
\draw[edge] (-4_0_1) -- (-3_-1_0);
\draw[edge] (-4_1_1) -- (-3_1_0);
\draw[edge] (-3_1_0) -- (-3_1_1);
\draw[edge] (-4_1_1) -- (-3_0_0);
\draw[periedge] (-4_2_1) -- (-3_2_0);
\draw[edge] (-3_2_0) -- (-3_2_1);
\draw[edge] (-4_2_1) -- (-3_1_0);
\draw[periedge] (-4_3_1) -- (-3_3_0);
\draw[periedge] (-3_3_0) -- (-3_3_1);
\draw[periedge] (-3_3_1) -- (-2_2_0);
\draw[periedge] (-4_3_1) -- (-3_2_0);
\draw[edge] (-5_0_1) -- (-4_0_0);
\draw[edge] (-4_0_0) -- (-4_0_1);
\draw[periedge] (-4_-1_0) -- (-4_-1_1);
\draw[periedge] (-5_0_1) -- (-4_-1_0);
\draw[periedge] (-5_1_1) -- (-4_1_0);
\draw[edge] (-4_1_0) -- (-4_1_1);
\draw[edge] (-5_1_1) -- (-4_0_0);
\draw[periedge] (-5_2_1) -- (-4_2_0);
\draw[periedge] (-4_2_0) -- (-4_2_1);
\draw[periedge] (-5_2_1) -- (-4_1_0);
\draw[periedge] (-6_1_1) -- (-5_1_0);
\draw[periedge] (-5_1_0) -- (-5_1_1);
\draw[periedge] (-5_0_0) -- (-5_0_1);
\draw[periedge] (-6_1_1) -- (-5_0_0);
\end{tikzpicture}
}
\quad
\subfigure[$S(3)$]{
\begin{tikzpicture}[scale=0.238,rotate=60]
\tikzstyle{every node} = [inner sep=1.2, draw, circle,fill=gray!60]
\tikzstyle{edge} = [draw, line width=1.0]
\tikzstyle{periedge} = [draw, line width=1.0]
\node (0_0_0) at (0.000000, 1.400000) {};
\node (0_0_1) at (1.212436, 0.700000) {};
\node (1_-1_0) at (1.212436, -0.700000) {};
\node (0_-1_1) at (0.000000, -1.400000) {};
\node (0_-1_0) at (-1.212436, -0.700000) {};
\node (-1_0_1) at (-1.212436, 0.700000) {};
\node (-1_-1_0) at (-3.637307, -0.700000) {};
\node (-1_-1_1) at (-2.424871, -1.400000) {};
\node (0_-2_0) at (-2.424871, -2.800000) {};
\node (-1_-2_1) at (-3.637307, -3.500000) {};
\node (-1_-2_0) at (-4.849742, -2.800000) {};
\node (-2_-1_1) at (-4.849742, -1.400000) {};
\node (-1_0_0) at (-2.424871, 1.400000) {};
\node (-2_0_1) at (-3.637307, 0.700000) {};
\node (-2_-2_0) at (-7.274613, -2.800000) {};
\node (-2_-2_1) at (-6.062178, -3.500000) {};
\node (-1_-3_0) at (-6.062178, -4.900000) {};
\node (-2_-3_1) at (-7.274613, -5.600000) {};
\node (-2_-3_0) at (-8.487049, -4.900000) {};
\node (-3_-2_1) at (-8.487049, -3.500000) {};
\node (-2_-1_0) at (-6.062178, -0.700000) {};
\node (-3_-1_1) at (-7.274613, -1.400000) {};
\node (-2_0_0) at (-4.849742, 1.400000) {};
\node (-3_0_1) at (-6.062178, 0.700000) {};
\node (-2_1_0) at (-3.637307, 3.500000) {};
\node (-2_1_1) at (-2.424871, 2.800000) {};
\node (-3_1_1) at (-4.849742, 2.800000) {};
\node (-3_-3_0) at (-10.911920, -4.900000) {};
\node (-3_-3_1) at (-9.699485, -5.600000) {};
\node (-2_-4_0) at (-9.699485, -7.000000) {};
\node (-3_-4_1) at (-10.911920, -7.700000) {};
\node (-3_-4_0) at (-12.124356, -7.000000) {};
\node (-4_-3_1) at (-12.124356, -5.600000) {};
\node (-3_-2_0) at (-9.699485, -2.800000) {};
\node (-4_-2_1) at (-10.911920, -3.500000) {};
\node (-3_-1_0) at (-8.487049, -0.700000) {};
\node (-4_-1_1) at (-9.699485, -1.400000) {};
\node (-3_0_0) at (-7.274613, 1.400000) {};
\node (-4_0_1) at (-8.487049, 0.700000) {};
\node (-3_1_0) at (-6.062178, 3.500000) {};
\node (-4_1_1) at (-7.274613, 2.800000) {};
\node (-3_2_0) at (-4.849742, 5.600000) {};
\node (-3_2_1) at (-3.637307, 4.900000) {};
\node (-4_2_1) at (-6.062178, 4.900000) {};
\node (-3_3_0) at (-3.637307, 7.700000) {};
\node (-3_3_1) at (-2.424871, 7.000000) {};
\node (-2_2_0) at (-2.424871, 5.600000) {};
\node (-4_3_1) at (-4.849742, 7.000000) {};
\node (-4_-1_0) at (-10.911920, -0.700000) {};
\node (-4_-2_0) at (-12.124356, -2.800000) {};
\node (-5_-1_1) at (-12.124356, -1.400000) {};
\node (-4_0_0) at (-9.699485, 1.400000) {};
\node (-5_0_1) at (-10.911920, 0.700000) {};
\node (-4_1_0) at (-8.487049, 3.500000) {};
\node (-5_1_1) at (-9.699485, 2.800000) {};
\node (-4_2_0) at (-7.274613, 5.600000) {};
\node (-5_2_1) at (-8.487049, 4.900000) {};
\node (-5_0_0) at (-12.124356, 1.400000) {};
\node (-5_-1_0) at (-13.336791, -0.700000) {};
\node (-6_0_1) at (-13.336791, 0.700000) {};
\node (-5_1_0) at (-10.911920, 3.500000) {};
\node (-6_1_1) at (-12.124356, 2.800000) {};
\node (-6_0_0) at (-14.549227, 1.400000) {};
\node (-6_-1_1) at (-14.549227, -1.400000) {};
\node (-6_-1_0) at (-15.761662, -0.700000) {};
\node (-7_0_1) at (-15.761662, 0.700000) {};
\draw[periedge] (-1_0_1) -- (0_0_0);
\draw[periedge] (0_0_0) -- (0_0_1);
\draw[periedge] (0_0_1) -- (1_-1_0);
\draw[periedge] (0_-1_1) -- (1_-1_0);
\draw[periedge] (0_-1_0) -- (0_-1_1);
\draw[edge] (-1_0_1) -- (0_-1_0);
\draw[edge] (-2_-1_1) -- (-1_-1_0);
\draw[edge] (-1_-1_0) -- (-1_-1_1);
\draw[periedge] (-1_-1_1) -- (0_-2_0);
\draw[periedge] (-1_-2_1) -- (0_-2_0);
\draw[periedge] (-1_-2_0) -- (-1_-2_1);
\draw[edge] (-2_-1_1) -- (-1_-2_0);
\draw[edge] (-2_0_1) -- (-1_0_0);
\draw[periedge] (-1_0_0) -- (-1_0_1);
\draw[periedge] (-1_-1_1) -- (0_-1_0);
\draw[edge] (-2_0_1) -- (-1_-1_0);
\draw[edge] (-3_-2_1) -- (-2_-2_0);
\draw[edge] (-2_-2_0) -- (-2_-2_1);
\draw[periedge] (-2_-2_1) -- (-1_-3_0);
\draw[periedge] (-2_-3_1) -- (-1_-3_0);
\draw[periedge] (-2_-3_0) -- (-2_-3_1);
\draw[edge] (-3_-2_1) -- (-2_-3_0);
\draw[edge] (-3_-1_1) -- (-2_-1_0);
\draw[edge] (-2_-1_0) -- (-2_-1_1);
\draw[periedge] (-2_-2_1) -- (-1_-2_0);
\draw[edge] (-3_-1_1) -- (-2_-2_0);
\draw[edge] (-3_0_1) -- (-2_0_0);
\draw[edge] (-2_0_0) -- (-2_0_1);
\draw[edge] (-3_0_1) -- (-2_-1_0);
\draw[edge] (-3_1_1) -- (-2_1_0);
\draw[periedge] (-2_1_0) -- (-2_1_1);
\draw[periedge] (-2_1_1) -- (-1_0_0);
\draw[edge] (-3_1_1) -- (-2_0_0);
\draw[periedge] (-4_-3_1) -- (-3_-3_0);
\draw[edge] (-3_-3_0) -- (-3_-3_1);
\draw[periedge] (-3_-3_1) -- (-2_-4_0);
\draw[periedge] (-3_-4_1) -- (-2_-4_0);
\draw[periedge] (-3_-4_0) -- (-3_-4_1);
\draw[periedge] (-4_-3_1) -- (-3_-4_0);
\draw[edge] (-4_-2_1) -- (-3_-2_0);
\draw[edge] (-3_-2_0) -- (-3_-2_1);
\draw[periedge] (-3_-3_1) -- (-2_-3_0);
\draw[periedge] (-4_-2_1) -- (-3_-3_0);
\draw[edge] (-4_-1_1) -- (-3_-1_0);
\draw[edge] (-3_-1_0) -- (-3_-1_1);
\draw[edge] (-4_-1_1) -- (-3_-2_0);
\draw[edge] (-4_0_1) -- (-3_0_0);
\draw[edge] (-3_0_0) -- (-3_0_1);
\draw[edge] (-4_0_1) -- (-3_-1_0);
\draw[edge] (-4_1_1) -- (-3_1_0);
\draw[edge] (-3_1_0) -- (-3_1_1);
\draw[edge] (-4_1_1) -- (-3_0_0);
\draw[periedge] (-4_2_1) -- (-3_2_0);
\draw[edge] (-3_2_0) -- (-3_2_1);
\draw[periedge] (-3_2_1) -- (-2_1_0);
\draw[edge] (-4_2_1) -- (-3_1_0);
\draw[periedge] (-4_3_1) -- (-3_3_0);
\draw[periedge] (-3_3_0) -- (-3_3_1);
\draw[periedge] (-3_3_1) -- (-2_2_0);
\draw[periedge] (-3_2_1) -- (-2_2_0);
\draw[periedge] (-4_3_1) -- (-3_2_0);
\draw[edge] (-5_-1_1) -- (-4_-1_0);
\draw[edge] (-4_-1_0) -- (-4_-1_1);
\draw[periedge] (-4_-2_0) -- (-4_-2_1);
\draw[periedge] (-5_-1_1) -- (-4_-2_0);
\draw[edge] (-5_0_1) -- (-4_0_0);
\draw[edge] (-4_0_0) -- (-4_0_1);
\draw[edge] (-5_0_1) -- (-4_-1_0);
\draw[periedge] (-5_1_1) -- (-4_1_0);
\draw[edge] (-4_1_0) -- (-4_1_1);
\draw[edge] (-5_1_1) -- (-4_0_0);
\draw[periedge] (-5_2_1) -- (-4_2_0);
\draw[periedge] (-4_2_0) -- (-4_2_1);
\draw[periedge] (-5_2_1) -- (-4_1_0);
\draw[periedge] (-6_0_1) -- (-5_0_0);
\draw[edge] (-5_0_0) -- (-5_0_1);
\draw[periedge] (-5_-1_0) -- (-5_-1_1);
\draw[edge] (-6_0_1) -- (-5_-1_0);
\draw[periedge] (-6_1_1) -- (-5_1_0);
\draw[periedge] (-5_1_0) -- (-5_1_1);
\draw[periedge] (-6_1_1) -- (-5_0_0);
\draw[periedge] (-7_0_1) -- (-6_0_0);
\draw[periedge] (-6_0_0) -- (-6_0_1);
\draw[periedge] (-6_-1_1) -- (-5_-1_0);
\draw[periedge] (-6_-1_0) -- (-6_-1_1);
\draw[periedge] (-7_0_1) -- (-6_-1_0);
\end{tikzpicture}
}
\quad
\subfigure[$T(3)$]{
\begin{tikzpicture}[scale=0.67,xscale=1,rotate=60]
\tikzstyle{every node} = [inner sep=1.2, draw, circle,fill=gray!60]
\tikzstyle{edge} = [draw, line width=1.0]
\tikzstyle{periedge} = [draw, line width=1.0]
\node (01) at (3.03,2.75) {};
\node (02) at (3.46,2.5) {};
\node (03) at (3.9,2.75) {};
\node (04) at (3.9,3.25) {};
\node (05) at (3.46,3.5) {};
\node (06) at (3.03,3.25) {};
\node (11) at (2.6,2) {};
\node (12) at (3.03,1.75) {};
\node (13) at (3.46,2) {};
\node (21) at (3.9,1.75) {};
\node (22) at (4.3,2) {};
\node (23) at (4.3,2.5) {};
\node (31) at (4.76,2.75) {};
\node (32) at (4.76,3.25) {};
\node (33) at (4.3,3.5) {};
\node (41) at (4.3,4) {};
\node (42) at (3.9,4.25) {};
\node (43) at (3.46,4) {};
\node (51) at (3.03,4.25) {};
\node (52) at (2.6,4) {};
\node (53) at (2.6,3.5) {};
\node (61) at (2.17,3.25) {};
\node (62) at (2.17,2.75) {};
\node (63) at (2.6,2.5) {};
\node (a1) at (3.03,1.25) {};
\node (a2) at (3.46,1) {};
\node (a3) at (3.9,1.25) {};
\node (b1) at (4.76,1.75) {};
\node (b2) at (5.2,2) {};
\node (b3) at (5.2,2.5) {};
\node (c1) at (5.2,3.5) {};
\node (c2) at (5.2,4) {};
\node (c3) at (4.76,4.25) {};
\node (d1) at (3.9,4.75) {};
\node (d2) at (3.46,5) {};
\node (d3) at (3.03,4.75) {};
\node (e1) at (2.17,4.25) {};
\node (e2) at (1.73,4) {};
\node (e3) at (1.73,3.5) {};
\node (f1) at (1.73,2.5) {};
\node (f2) at (1.73,2) {};
\node (f3) at (2.17,1.75) {};
\draw[edge] (01)--(02)--(03)--(04)--(05)--(06)--(01);
\draw[edge] (01)--(63)--(11)--(12)--(13)--(02);
\draw[edge] (02)--(13)--(21)--(22)--(23)--(03);
\draw[edge] (03)--(23)--(31)--(32)--(33)--(04);
\draw[edge] (04)--(33)--(41)--(42)--(43)--(05);
\draw[edge] (05)--(43)--(51)--(52)--(53)--(06);
\draw[edge] (06)--(53)--(61)--(62)--(63)--(01);
\draw[edge] (12)--(a1)--(a2)--(a3)--(21);
\draw[edge] (22)--(b1)--(b2)--(b3)--(31);
\draw[edge] (32)--(c1)--(c2)--(c3)--(41);
\draw[edge] (42)--(d1)--(d2)--(d3)--(51);
\draw[edge] (52)--(e1)--(e2)--(e3)--(61);
\draw[edge] (62)--(f1)--(f2)--(f3)--(11);
\end{tikzpicture}
}
\caption{Examples of the families defined in Table \ref{table:family}.}
\label{fig:examples}
\end{figure}
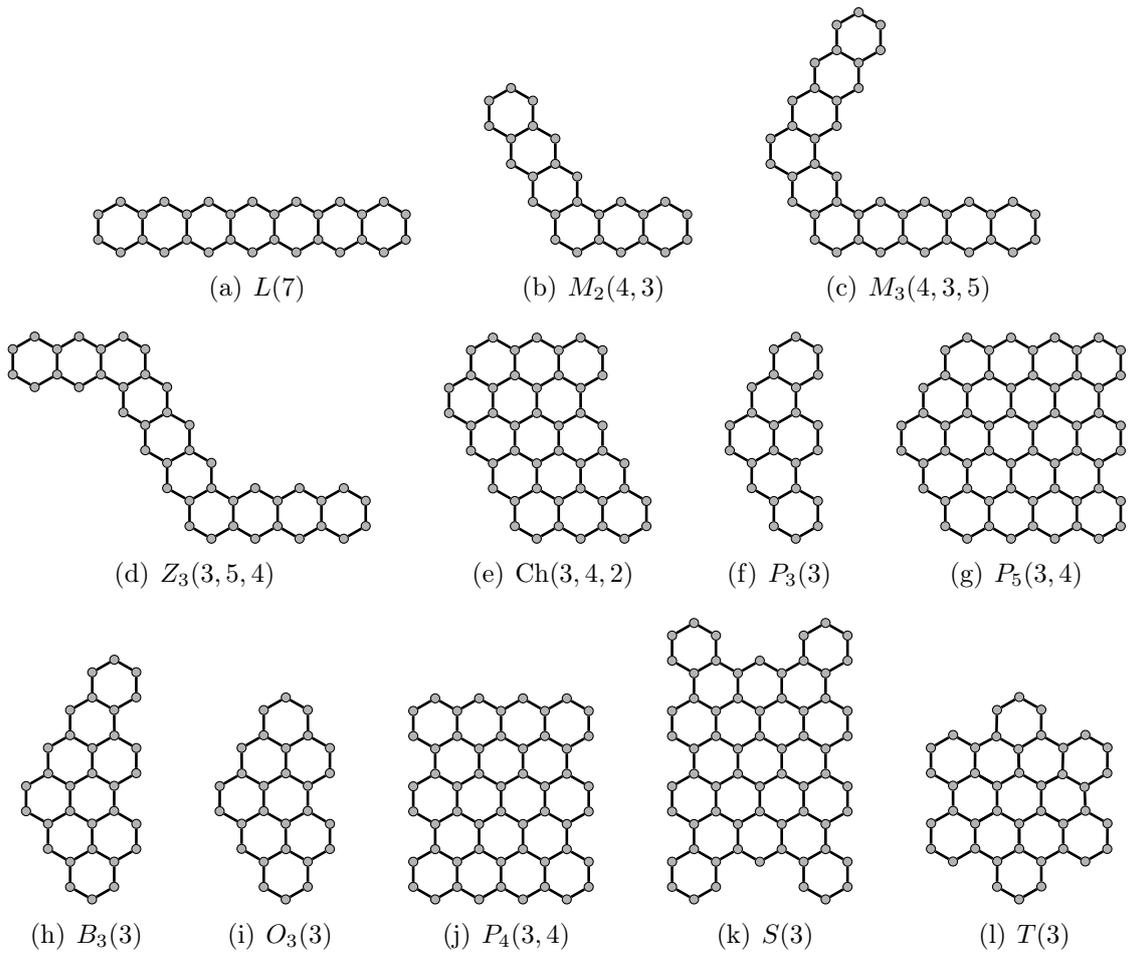

\section{Extremal convexity deficit} 

Clearly benzenoids with a small number of hexagons cannot have large
convexity deficit. For instance, benzenoids with up to $5$ hexagons have
convexity deficit at most $3$. 

\begin{definition}
  Let $B$ be a benzenoid. Let $\hex(B)$ denote the number of hexagons of
  $B$.  Let $\mcd(h)$ denote the maximum convexity deficit among all
  benzenoids on $h$ hexagons, i.e.
  \begin{equation*}
    \mcd(h) = \max \left\{ \cd(B) \mid \hex(B) = h \right\}.
  \end{equation*}
  Call each benzenoid attaining $\cd(h)$ \emph{extremal}, and $\ex(h)$ the
  number of extremal benzenoids with $h$ hexagons.
\end{definition}

We performed a computer search to find extremal benzenoids among all
benzenoids with $h$ hexagons for all $h \leq 18$.  The results are
summarised in Table~\ref{table:computerResults}. In particular, we noticed
that only one extremal benzenoid is pericondensed. It has $h = 6$ hexagons
and can be described by the boundary-edges code $533244111$ and is depicted
in Figure~\ref{fig:smallSpecialA}. Moreover, all other extremal benzenoids
with $h \leq 6$ are unbranched. There is a unique smallest branched
extremal benzenoid having $h = 7$. It has boundary-edges code
$523315151112$ and is depicted in Figure~\ref{fig:smallSpecialB}. Note that
the spiral benzenoid $S(h)$ attains the maximum value of convexity deficit
among all unbranched catacondensed benzenoids.  For $h \geq 14$, it appears
that all extremal benzenoids are branched.

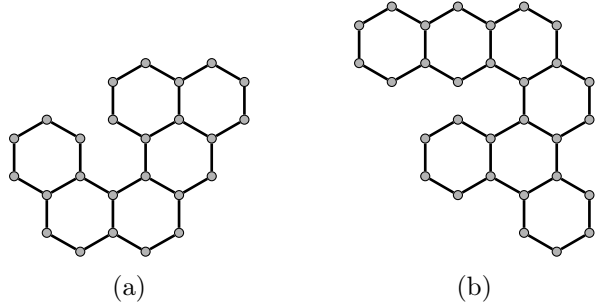
\begin{figure}[!htbp]
\centering
\subfigure[]{
\label{fig:smallSpecialA}
\begin{tikzpicture}[scale=1,xscale=1]
\tikzstyle{every node} = [inner sep=1.2, draw, circle,fill=gray!60]
\tikzstyle{edge} = [draw, line width=1.0]
\tikzstyle{periedge} = [draw, line width=1.0]
\node (h11) at (0.87,1) {};
\node (h12) at (0.43,0.75) {};
\node (h13) at (0,1) {};
\node (h14) at (0,1.5) {};
\node (h15) at (0.43,1.75) {};
\node (h16) at (0.87,1.5) {};
\node (h21) at (0.43,0.25) {};
\node (h22) at (0.87,0) {};
\node (h23) at (1.3,0.25) {};
\node (h24) at (1.3,0.75) {};
\node (h31) at (1.73,1) {};
\node (h32) at (2.17,0.75) {};
\node (h33) at (2.17,0.25) {};
\node (h34) at (1.73,0) {};
\node (h41) at (1.73,1.5) {};
\node (h42) at (2.17,1.75) {};
\node (h43) at (2.60,1.5) {};
\node (h44) at (2.60,1) {};
\node (h51) at (1.3,1.75) {};
\node (h52) at (1.3,2.25) {};
\node (h53) at (1.73,2.5) {};
\node (h54) at (2.17,2.25) {};
\node (h61) at (2.60,2.5) {};
\node (h62) at (3.03,2.25) {};
\node (h63) at (3.03,1.75) {};
\draw[edge] (h11)--(h12)--(h13)--(h14)--(h15)--(h16)--(h11);
\draw[edge] (h12)--(h21)--(h22)--(h23)--(h24)--(h11);
\draw[edge] (h24)--(h31)--(h32)--(h33)--(h34)--(h23);
\draw[edge] (h31)--(h41)--(h42)--(h43)--(h44)--(h32);
\draw[edge] (h41)--(h51)--(h52)--(h53)--(h54)--(h42);
\draw[edge] (h54)--(h61)--(h62)--(h63)--(h43);
\end{tikzpicture}
}
\qquad
\subfigure[]{
\label{fig:smallSpecialB}
\begin{tikzpicture}[scale=1,xscale=1]
\tikzstyle{every node} = [inner sep=1.2, draw, circle,fill=gray!60]
\tikzstyle{edge} = [draw, line width=1.0]
\tikzstyle{periedge} = [draw, line width=1.0]
\node (h11) at (2.6,0) {};
\node (h12) at (3.03,0.25) {};
\node (h13) at (3.03,0.75) {};
\node (h14) at (2.6,1) {};
\node (h15) at (2.17,0.75) {};
\node (h16) at (2.17,0.25) {};
\node (h21) at (2.6,1.5) {};
\node (h22) at (2.17,1.75) {};
\node (h23) at (1.73,1.5) {};
\node (h24) at (1.73,1) {};
\node (h31) at (1.3,1.75) {};
\node (h32) at (0.87,1.5) {};
\node (h33) at (0.87,1) {};
\node (h34) at (1.3,0.75) {};
\node (h41) at (2.17,2.25) {};
\node (h42) at (2.6,2.5) {};
\node (h43) at (3.03,2.25) {};
\node (h44) at (3.03,1.75) {};
\node (h51) at (2.6,3) {};
\node (h52) at (2.17,3.25) {};
\node (h53) at (1.73,3) {};
\node (h54) at (1.73,2.5) {};
\node (h61) at (1.3,3.25) {};
\node (h62) at (0.87,3) {};
\node (h63) at (0.87,2.5) {};
\node (h64) at (1.3,2.25) {};
\node (h71) at (0.43,3.25) {};
\node (h72) at (0,3) {};
\node (h73) at (0,2.5) {};
\node (h74) at (0.43,2.25) {};
\draw[edge] (h11)--(h12)--(h13)--(h14)--(h15)--(h16)--(h11);
\draw[edge] (h14)--(h21)--(h22)--(h23)--(h24)--(h15);
\draw[edge] (h24)--(h34)--(h33)--(h32)--(h31)--(h23);
\draw[edge] (h22)--(h41)--(h42)--(h43)--(h44)--(h21);
\draw[edge] (h42)--(h51)--(h52)--(h53)--(h54)--(h41);
\draw[edge] (h53)--(h61)--(h62)--(h63)--(h64)--(h54);
\draw[edge] (h62)--(h71)--(h72)--(h73)--(h74)--(h63);
\end{tikzpicture}
}
\caption{The only known pericondensed extremal benzenoid (a) and the
  smallest branched extremal benzenoid (b).}
\label{fig:smallSpecial}
\end{figure}

\begin{table}[!hb]
\centering
\caption{Maximal convexity deficit $\mcd(h)$ for each $2 \leq h \leq 18$
  and the number of extremal benzenoids $\ex(h)$. The last column contains
  an example of such a benzenoid.}
 \label{table:computerResults}
 \bigskip
 \begin{tabular}{| r r r l |}
  \hline
  $\bm{h}$  & $\bm{\mcd(h)}$ & $\bm{\ex(h)}$ & \textbf{An example (boundary-edges code)} \\
  \hline\hline
  2& 0  & 1  & 55 \\
  3 & 1  & 1  & 5351 \\
  4 & 2  & 2  & 532521 \\
  5 & 3  & 6  & 52325212\\
  6 & 4 & 16  & 5232252212\\
  7 & 6 & 3  & 523315151112\\
  8 & 8 & 2  & 53323325211211\\
  9 & 10 & 3   & 5332332252211211\\
  10 & 12 & 6 & 533233222522211211 \\
  11 & 14 & 16  & 52311121225223233312 \\
  12 & 16 & 37 & 5332332222252222211211 \\
  13 & 18 & 102  & 533233222222522222211211 \\
  14 & 21 & 2  & 53332322215133511122212111 \\
  15 & 23 & 12  & 5332332132151335111213211211 \\
  16 & 25 & 42  & 533323222321513351112122212111 \\
17 & 27 & 149  & 53323321323215133511121213211211 \\
18 & 29 & 489  & 5333232223232151335111212122212111 \\
  \hline
 \end{tabular}
\end{table}

We were able to find an interesting family of benzenoids, one member for
each number of hexagons. We call them \emph{spiral benzenoids}.

\begin{definition}
  Let $S(h)$ denote the spiral benzenoid on $h \geq 2$ hexagons determined by the
  following procedure.
  Let $a$ and $b$ denote the following infinite codes:
  \begin{align*}
    a & = 33323232322322322322232223 \ldots = \bigoplus\limits_{k=0}^\infty\, (2^k3)^3 \\
    b & = 11121212122122122122212221 \ldots = \bigoplus\limits_{k=0}^\infty\, (2^k1)^3,
  \end{align*}
  using $\bigoplus$ for repeated concatenation using the $\oplus$
  operation.  Let $a(\ell)$ denote the substring composed of the first
  $\ell$ symbols of $a$ and similarly define $b(\ell)$.  Let
  \begin{equation*}
    w(\ell) = 5a(\ell)5\rho b(\ell),
  \end{equation*}
  where $\rho b(\ell)$ denotes the reversal of $b(\ell)$.  The \emph{spiral
    benzenoid}, denoted $S(h)$, has $h$ hexagons and is defined by the
  boundary-edges code $w(h-2)$.
\end{definition}

For an example of $S(h)$, see Figure~\ref{fig:spiral}.
\begin{example}
For small values of $\ell$ we obtain:
\begin{align*}
S(2) = w(0) & = 55, \\
S(3) = w(1) & = 5352, \\
S(4) = w(2) & = 533511, \\
S(5) = w(3) & = 53335111, \\
S(6) = w(4) & = 5333252111,\text{ etc.}  
\end{align*}
All $S(2), \ldots, S(13)$ are extremal.
\end{example}

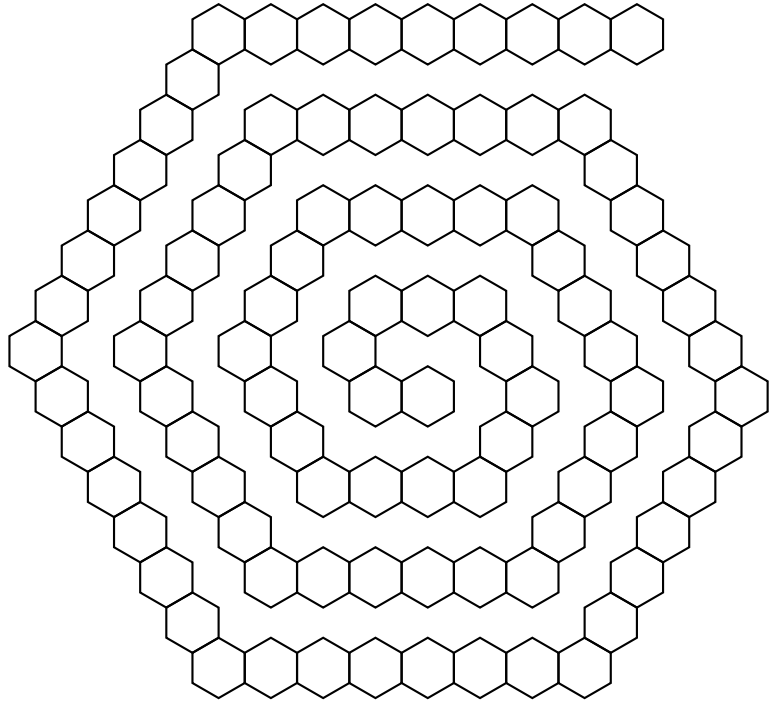
\begin{figure}[!htbp]
\centering
\begin{tikzpicture}[style=thick,scale=0.8]
\draw (8.17,4.75)--(8.6,5)--(8.6,5.5)--(8.17,5.75)--(7.74,5.5)--(7.74,5)--cycle;
\draw (7.31,4.75)--(7.74,5)--(7.74,5.5)--(7.31,5.75)--(6.88,5.5)--(6.88,5)--cycle;
\draw (6.88,5.5)--(7.31,5.75)--(7.31,6.25)--(6.88,6.5)--(6.45,6.25)--(6.45,5.75)--cycle;
\draw (7.31,6.25)--(7.74,6.5)--(7.74,7)--(7.31,7.25)--(6.88,7)--(6.88,6.5)--cycle;
\draw (8.17,6.25)--(8.6,6.5)--(8.6,7)--(8.17,7.25)--(7.74,7)--(7.74,6.5)--cycle;
\draw (9.03,6.25)--(9.46,6.5)--(9.46,7)--(9.03,7.25)--(8.6,7)--(8.6,6.5)--cycle;
\draw (9.46,5.5)--(9.89,5.75)--(9.89,6.25)--(9.46,6.5)--(9.03,6.25)--(9.03,5.75)--cycle;
\draw (9.89,4.75)--(10.32,5)--(10.32,5.5)--(9.89,5.75)--(9.46,5.5)--(9.46,5)--cycle;
\draw (9.46,4)--(9.89,4.25)--(9.89,4.75)--(9.46,5)--(9.03,4.75)--(9.03,4.25)--cycle;
\draw (9.03,3.25)--(9.46,3.5)--(9.46,4)--(9.03,4.25)--(8.6,4)--(8.6,3.5)--cycle;
\draw (8.17,3.25)--(8.6,3.5)--(8.6,4)--(8.17,4.25)--(7.74,4)--(7.74,3.5)--cycle;
\draw (7.31,3.25)--(7.74,3.5)--(7.74,4)--(7.31,4.25)--(6.88,4)--(6.88,3.5)--cycle;
\draw (6.45,3.25)--(6.88,3.5)--(6.88,4)--(6.45,4.25)--(6.02,4)--(6.02,3.5)--cycle;
\draw (6.02,4)--(6.45,4.25)--(6.45,4.75)--(6.02,5)--(5.59,4.75)--(5.59,4.25)--cycle;
\draw (5.59,4.75)--(6.02,5)--(6.02,5.5)--(5.59,5.75)--(5.16,5.5)--(5.16,5)--cycle;
\draw (5.16,5.5)--(5.59,5.75)--(5.59,6.25)--(5.16,6.5)--(4.73,6.25)--(4.73,5.75)--cycle;
\draw (5.59,6.25)--(6.02,6.5)--(6.02,7)--(5.59,7.25)--(5.16,7)--(5.16,6.5)--cycle;
\draw (6.02,7)--(6.45,7.25)--(6.45,7.75)--(6.02,8)--(5.59,7.75)--(5.59,7.25)--cycle;
\draw (6.45,7.75)--(6.88,8)--(6.88,8.5)--(6.45,8.75)--(6.02,8.5)--(6.02,8)--cycle;
\draw (7.31,7.75)--(7.74,8)--(7.74,8.5)--(7.31,8.75)--(6.88,8.5)--(6.88,8)--cycle;
\draw (8.17,7.75)--(8.60,8)--(8.60,8.5)--(8.17,8.75)--(7.74,8.5)--(7.74,8)--cycle;
\draw (9.03,7.75)--(9.46,8)--(9.46,8.5)--(9.03,8.75)--(8.60,8.5)--(8.60,8)--cycle;
\draw (9.89,7.75)--(10.32,8)--(10.32,8.5)--(9.89,8.75)--(9.46,8.5)--(9.46,8)--cycle;
\draw (10.32,7)--(10.75,7.25)--(10.75,7.75)--(10.32,8)--(9.89,7.75)--(9.89,7.25)--cycle;
\draw (10.75,6.25)--(11.18,6.5)--(11.18,7)--(10.75,7.25)--(10.32,7)--(10.32,6.5)--cycle;
\draw (11.18,5.5)--(11.61,5.75)--(11.61,6.25)--(11.18,6.5)--(10.75,6.25)--(10.75,5.75)--cycle;
\draw (11.61,4.75)--(12.04,5)--(12.04,5.5)--(11.61,5.75)--(11.18,5.5)--(11.18,5)--cycle;
\draw (11.18,4)--(11.61,4.25)--(11.61,4.75)--(11.18,5)--(10.75,4.75)--(10.75,4.25)--cycle;
\draw (10.75,3.25)--(11.18,3.5)--(11.18,4)--(10.75,4.25)--(10.32,4)--(10.32,3.5)--cycle;
\draw (10.32,2.5)--(10.75,2.75)--(10.75,3.25)--(10.32,3.5)--(9.89,3.25)--(9.89,2.75)--cycle;
\draw (9.89,1.75)--(10.32,2)--(10.32,2.5)--(9.89,2.75)--(9.46,2.5)--(9.46,2)--cycle;
\draw (9.03,1.75)--(9.46,2)--(9.46,2.5)--(9.03,2.75)--(8.60,2.5)--(8.6,2)--cycle;
\draw (8.17,1.75)--(8.60,2)--(8.60,2.5)--(8.17,2.75)--(7.74,2.5)--(7.74,2)--cycle;
\draw (7.31,1.75)--(7.74,2)--(7.74,2.5)--(7.31,2.75)--(6.88,2.5)--(6.88,2)--cycle;
\draw (6.45,1.75)--(6.88,2)--(6.88,2.5)--(6.45,2.75)--(6.02,2.5)--(6.02,2)--cycle;
\draw (5.59,1.75)--(6.02,2)--(6.02,2.5)--(5.59,2.75)--(5.16,2.5)--(5.16,2)--cycle;
\draw (5.16,2.5)--(5.59,2.75)--(5.59,3.25)--(5.16,3.5)--(4.73,3.25)--(4.73,2.75)--cycle;
\draw (4.73,3.25)--(5.16,3.5)--(5.16,4)--(4.73,4.25)--(4.3,4)--(4.3,3.5)--cycle;
\draw (4.3,4)--(4.73,4.25)--(4.73,4.75)--(4.3,5)--(3.87,4.75)--(3.87,4.25)--cycle;
\draw (3.87,4.75)--(4.3,5)--(4.3,5.5)--(3.87,5.75)--(3.44,5.5)--(3.44,5)--cycle;
\draw (3.44,5.5)--(3.87,5.75)--(3.87,6.25)--(3.44,6.5)--(3.01,6.25)--(3.01,5.75)--cycle;
\draw (3.87,6.25)--(4.3,6.5)--(4.3,7)--(3.87,7.25)--(3.44,7)--(3.44,6.5)--cycle;
\draw (4.3,7)--(4.73,7.25)--(4.73,7.75)--(4.3,8)--(3.87,7.75)--(3.87,7.25)--cycle;
\draw (4.73,7.75)--(5.16,8)--(5.16,8.5)--(4.73,8.75)--(4.3,8.5)--(4.3,8)--cycle;
\draw (5.16,8.5)--(5.59,8.75)--(5.59,9.25)--(5.16,9.5)--(4.73,9.25)--(4.73,8.75)--cycle;
\draw (5.59,9.25)--(6.02,9.5)--(6.02,10)--(5.59,10.25)--(5.16,10)--(5.16,9.5)--cycle;
\draw (6.45,9.25)--(6.88,9.5)--(6.88,10)--(6.45,10.25)--(6.02,10)--(6.02,9.5)--cycle;
\draw (7.31,9.25)--(7.74,9.5)--(7.74,10)--(7.31,10.25)--(6.88,10)--(6.88,9.5)--cycle;
\draw (8.17,9.25)--(8.60,9.5)--(8.60,10)--(8.17,10.25)--(7.74,10)--(7.74,9.5)--cycle;
\draw (9.03,9.25)--(9.46,9.5)--(9.46,10)--(9.03,10.25)--(8.60,10)--(8.60,9.5)--cycle;
\draw (9.89,9.25)--(10.32,9.5)--(10.32,10)--(9.89,10.25)--(9.46,10)--(9.46,9.5)--cycle;
\draw (10.75,9.25)--(11.18,9.5)--(11.18,10)--(10.75,10.25)--(10.32,10)--(10.32,9.5)--cycle;
\draw (11.18,8.5)--(11.61,8.75)--(11.61,9.25)--(11.18,9.5)--(10.75,9.25)--(10.75,8.75)--cycle;
\draw (11.61,7.75)--(12.04,8)--(12.04,8.5)--(11.61,8.75)--(11.18,8.5)--(11.18,8)--cycle;
\draw (12.04,7)--(12.47,7.25)--(12.47,7.75)--(12.04,8)--(11.61,7.75)--(11.61,7.25)--cycle;
\draw (12.47,6.25)--(12.90,6.5)--(12.90,7)--(12.47,7.25)--(12.04,7)--(12.04,6.5)--cycle;
\draw (12.90,5.5)--(13.33,5.75)--(13.33,6.25)--(12.90,6.5)--(12.47,6.25)--(12.47,5.75)--cycle;
\draw (13.33,4.75)--(13.76,5)--(13.76,5.5)--(13.33,5.75)--(12.9,5.5)--(12.9,5)--cycle;
\draw (12.90,4)--(13.33,4.25)--(13.33,4.75)--(12.90,5)--(12.47,4.75)--(12.47,4.25)--cycle;
\draw (12.47,3.25)--(12.90,3.5)--(12.90,4)--(12.47,4.25)--(12.04,4)--(12.04,3.5)--cycle;
\draw (12.04,2.5)--(12.47,2.75)--(12.47,3.25)--(12.04,3.5)--(11.61,3.25)--(11.61,2.75)--cycle;
\draw (11.61,1.75)--(12.04,2)--(12.04,2.5)--(11.61,2.75)--(11.18,2.5)--(11.18,2)--cycle;
\draw (11.18,1)--(11.61,1.25)--(11.61,1.75)--(11.18,2)--(10.75,1.75)--(10.75,1.25)--cycle;
\draw (10.75,0.25)--(11.18,0.5)--(11.18,1)--(10.75,1.25)--(10.32,1)--(10.32,0.5)--cycle;
\draw (9.89,0.25)--(10.32,0.5)--(10.32,1)--(9.89,1.25)--(9.46,1)--(9.46,0.5)--cycle;
\draw (9.03,0.25)--(9.46,0.5)--(9.46,1)--(9.03,1.25)--(8.60,1)--(8.6,0.5)--cycle;
\draw (8.17,0.25)--(8.60,0.5)--(8.60,1)--(8.17,1.25)--(7.74,1)--(7.74,0.5)--cycle;
\draw (7.31,0.25)--(7.74,0.5)--(7.74,1)--(7.31,1.25)--(6.88,1)--(6.88,0.5)--cycle;
\draw (6.45,0.25)--(6.88,0.5)--(6.88,1)--(6.45,1.25)--(6.02,1)--(6.02,0.5)--cycle;
\draw (5.59,0.25)--(6.02,0.5)--(6.02,1)--(5.59,1.25)--(5.16,1)--(5.16,0.5)--cycle;
\draw (4.73,0.25)--(5.16,0.5)--(5.16,1)--(4.73,1.25)--(4.3,1)--(4.3,0.5)--cycle;
\draw (4.3,1)--(4.73,1.25)--(4.73,1.75)--(4.3,2)--(3.87,1.75)--(3.87,1.25)--cycle;
\draw (3.87,1.75)--(4.3,2)--(4.3,2.5)--(3.87,2.75)--(3.44,2.5)--(3.44,2)--cycle;
\draw (3.44,2.5)--(3.87,2.75)--(3.87,3.25)--(3.44,3.5)--(3.01,3.25)--(3.01,2.75)--cycle;
\draw (3.01,3.25)--(3.44,3.5)--(3.44,4)--(3.01,4.25)--(2.58,4)--(2.58,3.5)--cycle;
\draw (2.58,4.)--(3.01,4.25)--(3.01,4.75)--(2.58,5)--(2.15,4.75)--(2.15,4.25)--cycle;
\draw (2.15,4.75)--(2.58,5)--(2.58,5.5)--(2.15,5.75)--(1.72,5.5)--(1.72,5)--cycle;
\draw (1.72,5.5)--(2.15,5.75)--(2.15,6.25)--(1.72,6.5)--(1.29,6.25)--(1.29,5.75)--cycle;
\draw (2.15,6.25)--(2.58,6.5)--(2.58,7)--(2.15,7.25)--(1.72,7)--(1.72,6.5)--cycle;
\draw (2.58,7)--(3.01,7.25)--(3.01,7.75)--(2.58,8)--(2.15,7.75)--(2.15,7.25)--cycle;
\draw (3.01,7.75)--(3.44,8)--(3.44,8.5)--(3.01,8.75)--(2.58,8.5)--(2.58,8)--cycle;
\draw (3.44,8.5)--(3.87,8.75)--(3.87,9.25)--(3.44,9.5)--(3.01,9.25)--(3.01,8.75)--cycle;
\draw (3.87,9.25)--(4.3,9.5)--(4.3,10)--(3.87,10.25)--(3.44,10)--(3.44,9.5)--cycle;
\draw (4.3,10)--(4.73,10.25)--(4.73,10.75)--(4.3,11)--(3.87,10.75)--(3.87,10.25)--cycle;
\draw (4.73,10.75)--(5.16,11)--(5.16,11.5)--(4.73,11.75)--(4.3,11.5)--(4.3,11)--cycle;
\draw (5.59,10.75)--(6.02,11)--(6.02,11.5)--(5.59,11.75)--(5.16,11.5)--(5.16,11)--cycle;
\draw (6.45,10.75)--(6.88,11)--(6.88,11.5)--(6.45,11.75)--(6.02,11.5)--(6.02,11)--cycle;
\draw (7.31,10.75)--(7.74,11)--(7.74,11.5)--(7.31,11.75)--(6.88,11.5)--(6.88,11)--cycle;
\draw (8.17,10.75)--(8.60,11)--(8.60,11.5)--(8.17,11.75)--(7.74,11.5)--(7.74,11)--cycle;
\draw (9.03,10.75)--(9.46,11)--(9.46,11.5)--(9.03,11.75)--(8.60,11.5)--(8.60,11)--cycle;
\draw (9.89,10.75)--(10.32,11)--(10.32,11.5)--(9.89,11.75)--(9.46,11.5)--(9.46,11)--cycle;
\draw (10.75,10.75)--(11.18,11)--(11.18,11.5)--(10.75,11.75)--(10.32,11.5)--(10.32,11)--cycle;
\draw (11.61,10.75)--(12.04,11)--(12.04,11.5)--(11.61,11.75)--(11.18,11.5)--(11.18,11)--cycle;
\end{tikzpicture}
 \label{fig:spiral}
 \caption{Spiral benzenoid $S(h)$ on $h$ hexagons has $\cd(S(h)) = \max \{ h-2,  2h-8\}$.}
\end{figure}

It is easy to find the extremal fusenes in the subclass $\classFnb$. A
small example is [6]helicene in Figure~\ref{fig:6helicene}.
\begin{proposition}
  In the class of unbranched catacondensed fusenes $\classFnb$, the
  helicene $51^{h-2}53^{h-2}$ obtains the maximal convexity deficit among
  all $F \in \classFnb$ with $h \geq 2$ hexagons. The maximal convexity
  deficit is $\max\{2h - 7, h - 2\}$.
\end{proposition}
\begin{proof}
  Each unbranched catacondensed fusene can be described with a
  boundary-edges code
  $5s_1 s_2 \ldots s_{h-2} 5 \bar{s}_1 \bar{s}_2 \ldots \bar{s}_{h-2}$,
  where $s_{i} + \bar{s}_{i} = 4$ for all $1 \leq i \leq h - 2$. It is
  clear that by setting $s_i = 1$ for all $i = 1, \ldots, h - 2$ we will
  obtain one of the fusenes, let us denote it by $F$, with the longest
  possible subcode $c$ such that $\frac{\suma(c)}{\len(c)} < 2$.

  If $h$ is large enough, then code will contain $h-2$ symbols $1$, symbol
  $5$ and a certain number, let us denote it by $\ell$, of symbols $3$. We
  are looking for the largest possible $\ell$ such that
  \begin{equation}
    \label{pogoj}
    \frac{(h - 2) + 5 + 3\ell}{h - 2 + 1 + \ell} < 2.
  \end{equation}
  Equation \eqref{pogoj} is equivalent to
  \begin{equation}
    \label{pogoj2}
    \ell < h - 5
  \end{equation}
  when $h + \ell > 3$ (this holds for large enough $h$ since
  $\ell \geq 0$). From Equation \eqref{pogoj2} it follows that we can take
  $\ell = h - 6$. This is valid if $h \geq 6$ and the convexity deficit
  equals $(h - 2) + 1 + (h - 6) = 2h - 7$.

  If $h < 6$, then there are only $4$ fusenes to analyse. By manual
  inspection we can see that the convexity deficit in each case is $h - 2$.
\end{proof}

\begin{proposition}
  Among all catacondensed unbranched benzenoids on $h$ hexagons, $S(h)$ has
  the maximal convexity deficit.  The convexity deficit of $S(h)$ is
  $\cd(S(h)) = \max\{ h-2, 2h - 8 \}$.  
\end{proposition}
\begin{proof}
  First, observe that the code $533323$ is a subcode of $S(h)$ for all
  $h \geq 7$.

  We know that $\suma(S(h)) 
  = 4h + 2$ and $\len(S(h)) 
  = 2h - 2$.

  Let $c$ be the subcode of $\code(S(h))$ that is obtaned from
  $\code(S(h))$ by erasing $533323$ (i.e.\ $c = 11121\ldots$). Then
  $\suma(c) = 4h + 2 - 19 = 4h - 17$ and $\len(c) = 2h - 2 - 6 =
  2h-8$. Therefore
  \begin{equation*}
    \frac{\suma(c)}{\len(c)} = \frac{4h-17}{2h-8} < 2.
  \end{equation*}
  It is easy to see that the above also holds for all prefixes of the code
  $c$. This implies that $\cd(S(h)) \geq 2h - 8$ for all $h \geq 7$. Let us
  now show that there exists no unbranched benzenoid $B$ such that
  $cd(B) = 2h-7$.

  For contradiction, suppose that there exists an unbranched benzenoid $B$
  such that $\cd(B) = 2h - 7$.  Let $c$ be the code which remains when the
  maximal subcode $d$, for which $\frac{\suma(d)}{\len(d)} < 2$, is erased
  from $\code(B)$. We have $\len(c) = 5$ and
  \begin{equation*}
    \frac{4h + 2 - \suma(c)}{2h - 7} < 2.
  \end{equation*}
  If $h$ is large enough, we obtain $\suma(c) > 16$.  The code $c$ contains
  $5$ symbols, each of which is an element of the set $\{1, 2, 3, 5\}$.
  One of the symbols must be $5$ (otherwise the sum can not be greater than
  $16$). Also, the code cannot contain symbol $5$ twice if the benzenoid
  is large enough (their corresponding hexagons are located at the opposite
  ends of the chain). From $\suma(c) > 16$ it follows that all the other
  symbols have to be $3$. The code $353$ can not be a subcode of a
  benzenoid due to geometric restrictions. The only remaining option is
  $c = 53333$, which again can not be a subcode of a benzenoid, a
  contradiction.

  Therefore, $S(h)$ attains the maximal convexity deficit among all
  unbranched benzenoids on $h$ hexagons.
\end{proof}

\subsection{Conclusion}

In this contribution we have briefly revisited several families of
benzenoids that have been studied in the past.  Most of them are taken from
the book by Cyvin and Gutman~\cite[p.~62]{cyvin_1988}.  Here they are
defined rigorously by the boundary-edges code instead of relying on
pictorial representations.
We considered extremal benzenoids with respect to convexity deficit.
Table~\ref{table:computerResults} summarises all small cases up to $18$ hexagons.
BECs of these benzenoids are stored in \cite{gitHubNino}. We observed 
from these data that
some clear patterns emerged.

In particular, let $F(h, k)$ denote the number of
  benzenoids on $h$ hexagons having convexity deficit equal to $k$. 
Note that $F(h, \mcd(h)) = \ex(h)$ and $F(h, k) = 0$ for all $k > \mcd(h)$.
\begin{conjecture}
Let $h \geq 0$.
The sequence 
$$
F(h, 0), F(h, 1), F(h, 2), \ldots, F(h, \mcd(h))
$$
is unimodal.
\end{conjecture}
The motivation for the previous statement comes from computational investigation
for small values of $h$.

Our empirical studies show an interesting picture of extremal
benzenoids. It seems that:
\begin{enumerate}[label=(\arabic*)]
\item there is only one extremal benzenoid that is pericondensed;
\item there are only finitely many extremal unbranched pericondensed
  benzenoids;
\item all extremal benzenoids for $h \geq 14$ are branched;
\item there is no upper bound on the number of branched points of extremal
  benzenoids when $h$ tends to infinity.
\end{enumerate}
These observations could be formulated as conjectures and
are a subject of further research.

\section*{Acknowledgements}

This work was supported in part by the Slovenian Research Agency (grants P1-0294, N1-032, 
and J1-7051), the German Research Foundation (grant
STA 850/19-2 within SPP 1738), and bilateral Slovenian-German project
``Mathematical Foundations of Selected Topics in Science''.

This paper was conceived at the Workshop on Discrete and Computational
Biomathematics, and Mathematical Chemistry, Koper, Slovenia, 15 -- 17 November
2017, organised by the Slovenian Discrete and Applied Mathematics Society
and University of Primorska, FAMNIT, and completed at the 33rd TBI
Winterseminar, Bled, Slovenia, 11 -- 17 February 2018.

\bibliographystyle{amcjoucc}
\bibliography{convA}

\begin{thebibliography}{10}
\expandafter\ifx\csname urlstyle\endcsname\relax
  \providecommand{\doi}[1]{doi:\discretionary{}{}{}#1}\else
  \providecommand{\doi}{doi:\discretionary{}{}{}\begingroup
  \urlstyle{rm}\Url}\fi

\bibitem{abdel2016}
H.~I. Abdel-Shafy and M.~S.~M. Mansour, A review on polycyclic aromatic
  hydrocarbons: {S}ource, environmental impact, effect on human health and
  remediation, \emph{Egypt. J. Petrol.} \textbf{25} (2016), 107--123,
  \doi{10.1016/j.ejpe.2015.03.011}.

\bibitem{allamandola1989}
L.~J. Allamandola, A.~G. G.~M. Tielens and J.~R. Barker, Interstellar
  polycyclic aromatic hydrocarbons: the infrared emission bands, the
  excitation/emission mechanism, and the astrophysical implications,
  \emph{Astrophys. J. Suppl. S.} \textbf{71} (1989), 733--775,
  \doi{10.1086/191396}.

\bibitem{BalKauKosBal1980}
K.~Balasubramanian, J.~J. Kaufman, W.~S. Koski and A.~T. Balaban, Graph
  theoretical characterization and computer-generation of certain carcinogenic
  benzenoid hydrocarbons and identification of bay regions, \emph{J. Comput.
  Chem.} \textbf{1} (1980), 149--157, \doi{10.1002/jcc.540010207}.

\bibitem{Bas2017}
N.~Ba{\v s}i{\'c}, Infinite benzenoids, \emph{Art Discrete Appl. Math.}
  \textbf{2} (2019), \#1.09, \doi{10.26493/2590-9770.1228.eb5}.

\bibitem{BasFowPis2017}
N.~Ba{\v s}i{\'c}, P.~W. Fowler and T.~Pisanski, Stratified enumeration of
  convex benzenoids, \emph{MATCH Commun. Math. Comput. Chem.} \textbf{80}
  (2018), 153--172,
  \url{http://match.pmf.kg.ac.rs/electronic_versions/Match80/n1/match80n1_153-172.pdf}.

\bibitem{gitHubNino}
N.~Bašić, Convexity {D}eficit {D}ata ({G}it{H}ub repository),
  \url{https://github.com/nbasic/convexity-deficit-data}.

\bibitem{Brinkmann2002}
G.~Brinkmann, G.~Caporossi and P.~Hansen, A constructive enumeration of fusenes
  and benzenoids, \emph{J. Algorithms} \textbf{45} (2002), 155--166,
  \doi{10.1016/s0196-6774(02)00215-8}.

\bibitem{Brinkmann2007}
G.~Brinkmann, C.~Grothaus and I.~Gutman, Fusenes and benzenoids with perfect
  matchings, \emph{J. Math. Chem.} \textbf{42} (2007), 909--924,
  \doi{10.1007/s10910-006-9148-z}.

\bibitem{clar1964}
E.~Clar, \emph{Polycyclic {H}ydrocarbons, {V}olume 1}, Springer-Verlag, Berlin,
  1st edition, 1964, \doi{10.1007/978-3-662-01665-7}.

\bibitem{clar1964a}
E.~Clar, \emph{Polycyclic {H}ydrocarbons, {V}olume 2}, Springer-Verlag, Berlin,
  1st edition, 1964, \doi{10.1007/978-3-662-01668-8}.

\bibitem{CruGutRad12}
R.~Cruz, I.~Gutman and J.~Rada, Convex hexagonal systems and their topological
  indices, \emph{MATCH Commun. Math. Comput. Chem.} \textbf{68} (2012),
  97--108,
  \url{http://match.pmf.kg.ac.rs/electronic_versions/Match68/n1/match68n1_97-108.pdf}.

\bibitem{cyvin1994}
S.~J. Cyvin, J.~Brunvoll, R.~S. Chen, B.~N. Cyvin and F.~J. Zhang, \emph{Theory
  of Coronoid Hydrocarbons {II}}, volume~62 of \emph{Lecture Notes in
  Chemistry}, Springer-Verlag, 1994.

\bibitem{cyvin1991}
S.~J. Cyvin, J.~Brunvoll and B.~N. Cyvin, \emph{Theory of Coronoid
  Hydrocarbons}, volume~54 of \emph{Lecture Notes in Chemistry},
  Springer-Verlag, 1991.

\bibitem{cyvin_1985}
S.~J. Cyvin, B.~N. Cyvin and I.~Gutman, Number of {K}ekul{\'e} structures of
  five-tier strips, \emph{Z. Naturforsch. A} \textbf{40} (1985), 1253--1261,
  \doi{10.1515/zna-1985-1211}.

\bibitem{cyvin_1988}
S.~J. Cyvin and I.~Gutman, \emph{Kekul{\'e} {S}tructures in {B}enzenoid
  {H}ydrocarbons}, volume~46 of \emph{Lecture Notes in Chemistry}, Springer,
  Heidelberg, 1988, \doi{10.1007/978-3-662-00892-8}.

\bibitem{cyvin1988}
S.~J. Cyvin and I.~Gutman, \emph{Kekulé Structures in Benzenoid Hydrocarbons},
  volume~48 of \emph{Lecture Notes in Chemistry}, Springer-Verlag, 1988.

\bibitem{PenaGutmanRada2007}
J.~A. de~la Pe{\~n}a, I.~Gutman and J.~Rada, Estimating the {E}strada index,
  \emph{Linear Algebra Appl.} \textbf{427} (2007), 70--76,
  \doi{10.1016/j.laa.2007.06.020}.

\bibitem{deza1990}
M.~Deza, P.~W. Fowler and V.~Grishukhin, Allowed boundary sequences for fused
  polycyclic patches, and related problems, \emph{J. Chem. Inf. Comput. Sci.}
  \textbf{41} (2001), 300--308, \doi{10.1021/ci000060o}.

\bibitem{dias1987}
J.~R. Dias, \emph{Handbook of {P}olycyclic {H}ydrocarbons: {P}art {A},
  {B}enzenoid {H}ydrocarbons}, volume 30A of \emph{Physical Sciences Data},
  Elsevier, Amsterdam, 1987.

\bibitem{dias1988}
J.~R. Dias, \emph{Handbook of {P}olycyclic {H}ydrocarbons: {P}art {B},
  {P}olycyclic {I}somers and {H}eteroatom {A}nalogs of {B}enzenoid
  {H}ydrocarbons}, volume 30B of \emph{Physical Sciences Data}, Elsevier,
  Amsterdam, 1988.

\bibitem{Dias2005}
J.~R. Dias, Perimeter topology of benzenoid polycyclic hydrocarbons,
  \emph{Journal of Chemical Information and Modeling} \textbf{45} (2005),
  562--571, \doi{10.1021/ci0500334}.

\bibitem{Dresselhaus}
M.~S. Dresselhaus, G.~Dresselhaus and P.~C. Eklund, \emph{Science of Fullerenes
  and Carbon Nanotubes}, Academic Press, London, 1996,
  \doi{10.1016/b978-0-12-221820-0.x5000-x}.

\bibitem{fetzer2000}
J.~C. Fetzer, \emph{Large ({C$_{\ge 24}$}) {P}olycyclic {A}romatic
  {H}ydrocarbons: {C}hemistry and {A}nalysis}, volume 158 of \emph{{C}hemical
  {A}nalysis: {A} {S}eries of {M}onographs in {A}nalytical {C}hemistry and
  {I}ts {A}pplications}, Wiley–Interscience, New York, 2000.

\bibitem{zhang_1986}
Z.~Fuji and C.~Rongsi, On {Kekul{\'e}} structure count of hexagonal systems,
  \emph{Journal of Xinjiang University (Natural Science Edition)} \textbf{1986}
  (1986), 10--15.

\bibitem{Gardner1978}
M.~Gardner, \emph{Mathematical {M}agic Show: More Puzzles, Games, Diversions,
  Illusions and Other Mathematical Sleight-of-Mind from Scientific American},
  Vintage, 1978.

\bibitem{Garratt1986}
P.~J. Garratt, \emph{Aromaticity}, Wiley, New York, 1986.

\bibitem{gelboin1978}
H.~V. Gelboin and P.~O.~P. Ts'o (eds.), \emph{Polycyclic {H}ydrocarbons and
  {C}ancer, {V}olume {1}: {E}nvironment, {C}hemistry, and {M}etabolism},
  Academic Press, 1978, \doi{10.1016/b978-0-12-279201-4.x5001-4}.

\bibitem{gordon_1952}
M.~Gordon and W.~H.~T. Davison, Theory of resonance topology of fully aromatic
  hydrocarbons. {I}, \emph{J. Chem. Phys.} \textbf{20} (1952), 428--435,
  \doi{10.1063/1.1700437}.

\bibitem{guo2002}
X.~Guo, P.~Hansen and M.~Zheng, Boundary uniqueness of fusenes, \emph{Disc.
  Appl. Math.} \textbf{118} (2002), 209--222,
  \doi{10.1016/s0166-218x(01)00180-9}.

\bibitem{gutman1989}
I.~Gutman and S.~J. Cyvin, \emph{Introduction to the {T}heory of {B}enzenoid
  {H}ydrocarbons}, Springer-Verlag, Heidelberg, 1989,
  \doi{10.1007/978-3-642-87143-6}.

\bibitem{hansen1996}
P.~Hansen, C.~Lebatteux and M.~Zheng, The boundary-edges code for polyhexes,
  \emph{J. Mol. Struct. (Theochem)} \textbf{363} (1996), 237--247,
  \doi{10.1016/0166-1280(95)04139-7}.

\bibitem{Jerina1980}
D.~M. Jerina, J.~M. Sayer, D.~R. Thakker, H.~Yagi, W.~Levin, A.~W. Wood and
  A.~H. Conney, Carcinogenicity of polycyclic aromatic hydrocarbons: The
  {Bay-Region Theory}, in: \emph{Carcinogenesis: Fundamental Mechanisms and
  Environmental Effects}, Springer, Dordrecht, volume~13 of \emph{The Jerusalem
  Symposia on Quantum Chemistry and Biochemistry}, 1980 pp. 1--12.

\bibitem{KnoSzyJerTri1983}
J.~V. Knop, K.~Szymanski, {\v Z}.~Jeri{\v c}evi{\'c} and N.~Trinajsti{\'c},
  Computer enumeration and generation of benzenoid hydrocarbons and
  identification of bay regions, \emph{J. Comput. Chem.} \textbf{4} (1983),
  23--32, \doi{10.1002/jcc.540040105}.

\bibitem{lloyd1989}
D.~Lloyd, \emph{The {C}hemistry of {C}onjugated {C}yclic {C}ompounds: {T}o {B}e
  or {N}ot to {B}e {L}ike {B}enzene?}, John Wiley \& Sons, 1989.

\bibitem{Murrell}
J.~N. Murrell, \emph{The {T}heory of the {E}lectronic {S}pectra of {O}rganic
  {M}olecules}, Methuen, London, 1963.

\bibitem{Rouvray2002}
D.~H. Rouvray, Chapter 2 -- {T}he rich legacy of half a century of the {W}iener
  index, in: D.~H. Rouvray and R.~B. King (eds.), \emph{Topology in
  {C}hemistry: {D}iscrete {M}athematics of {M}olecules}, Woodhead Publishing,
  Chichester, West Sussex, pp. 16--37, 2002, \doi{10.1533/9780857099617.16}.

\bibitem{odour1957}
{Society of Chemical Industry, London}, \emph{{M}olecular {S}tructure and
  {O}rganoleptic {Q}uality}, number~1 in S. C. I. {M}onograph Series,
  Macmillan, New York, 1958, symposium organized by the Overseas Section,
  Geneva, May 1957.

\bibitem{taylor1990}
R.~J. Taylor, \emph{Electrophilic {A}romatic {S}ubstitution}, Wiley,
  Chichester, 1990.

\bibitem{tomlinson1971}
M.~Tomlinson, \emph{An {I}ntroduction to the {C}hemistry of {B}enzenoid
  {C}ompounds}, Elsevier, New York, 1971, \doi{10.1016/c2013-0-10050-5}.

\bibitem{trinajstic1992}
N.~Trinajsti{\'c}, \emph{Chemical {G}raph {T}heory}, CRC Press, Boca Raton, 2nd
  edition, 1992.

\bibitem{yen_1971}
T.~F. Yen, Resonance topology of polynuclear aromatic hydrocarbons,
  \emph{Theor. Chem. Acc.} \textbf{20} (1971), 399--404,
  \doi{10.1007/bf00527196}.

\end{thebibliography}

\end{document}